\documentclass[journal]{IEEEtran}
\usepackage{float}
\usepackage{subfigure}
\usepackage[final]{graphicx}
\usepackage{amsthm}
\usepackage{amsmath}

\usepackage{enumitem}
\usepackage[nospace,noadjust]{cite}
\usepackage{textcomp}
\usepackage{amssymb}
\usepackage{hyperref}
\usepackage{mathrsfs}
\usepackage{comment}
\newtheorem{thm}{Theorem}[section]
\newtheorem{lem}[thm]{Lemma}
\newtheorem{cor}[thm]{Corollary}
\newtheorem{rmk}{Remark}
\newtheorem{defn}{Definition}
\ifCLASSINFOpdf
\else
\fi
\begin{document}
\title{Dynamics of stochastic approximation 
with iterate-dependent Markov noise under verifiable conditions in compact state space with the stability of iterates not ensured}

\author{Prasenjit Karmakar and Shalabh Bhatnagar% <-this % stops a space
\thanks{Prasenjit Karmakar is with the Department of Computer Science and
Automation, Indian Institute of Science Bangalore (e-mail: prasenjit@iisc.ac.in, Contact No:
+917337617564.}
\thanks{Shalabh Bhatnagar is with the Department of Computer Science and
Automation and the Robert Bosch Centre for Cyber Physical Systems, Indian Institute of Science Bangalore
 (e-mail: shalabh@iisc.ac.in, Phone: 91(80) 2293-2987 Fax:
91(80) 2360-2911)}}
%of Electrical and Computer Engineering, Georgia Institute of Technology, Atlanta,


\markboth{}%
{Prasenjit \MakeLowercase{\textit{et al.}}: Dynamics of stochastic approximation 
with Markov iterate-dependent noise under verifiable conditions with the stability of the iterates not ensured }



\maketitle

\begin{abstract}
This paper compiles several aspects of the dynamics  
%both asymptotic and non-asymptotic properties
of stochastic approximation algorithms with \textit{Markov iterate-dependent} noise when 
the iterates are not known to be stable beforehand. We achieve the same by
extending the  \textit{lock-in probability} (i.e. 
the probability of
convergence of the iterates to a specific attractor of the limiting o.d.e.\  given
that the iterates are in its domain of attraction
after a sufficiently large number of iterations (say) $n_0$) framework to such recursions. 
Specifically, with the more restrictive assumption of Markov iterate-dependent 
noise supported on a bounded subset of the Euclidean space 
we give a lower bound for the \textit{lock-in probability}.
%With the more restrictive assumption of Markov iterate-dependent noise supported on a bounded subset of the Euclidean space,
%we recover the existing bound
%in literature for the case of  martingale difference noise.
We use these results to prove almost sure convergence of the iterates to the specified attractor
when the iterates satisfy an asymptotic tightness condition. The novelty of our approach is that if the state space of the Markov process is compact we prove almost sure convergence under much weaker assumptions compared to the work by Andrieu et al. which solves the general state space case under much restrictive assumptions. 
We also extend our single timescale results to the case where there are two separate recursions over two different timescales. This, in turn, is shown to be useful
in analyzing the \textit{tracking ability} of general adaptive algorithms.    
%To the best of our knowledge this is the first time 
%an almost sure  convergence proof for such recursion is presented 
%without assuming the stability of the iterates, however following the classic Poisson equation
%model of Metivier and Priouret which is designed keeping 
%in mind the stability  of the iterates. 
%Another important corollary of our results is
%the stability and henceforth a.s. convergence of the iterates for some common step-size sequences
%if the iterates belong to some special open set with compact closure in the domain of attraction of the local attractor
%infinitely often with probability 1.
Additionally, we show that our results can be used to derive a \textit{sample complexity estimate} of such recursions,
which then can be used for step-size selection.
\end{abstract}


% \fi
\IEEEpeerreviewmaketitle

\section{Introduction}

Stochastic approximation algorithms are sequential
non-parametric methods for finding a zero or minimum of a function in cases where
only the noisy observations of the function values are available. Stochastic approximation
iterates in $\mathbb{R}^d$ are given by
\begin{align}
\label{eqn1}
\theta_{n+1} = \theta_n + a(n)f(\theta_n,Y_{n+1}), n\geq 0,
\end{align}
where $\theta_0$ is the initial point, $\{\theta_n\}$ are the iterates,
%$h:\mathbb{R}^d \to \mathbb{R}^d$ is the drift function,
$\{Y_n\}$ is an $\mathbb{R}^d$-
valued `Markov iterate-dependent' noise sequence, i.e., it  satisfies
\begin{align}
P[Y_{n+1} \in A | \mathcal{F}_n] = \int_A\Pi_{\theta_n}(Y_n; dx) \ \mbox{a.s.,}~ \forall n, \label{eqn4}
\end{align}
where $\mathcal{F}_n :=$ the $\sigma$-field generated by all random variables realized till time $n$ 
and $a(n)$ is the $n$-th step-size.

It is well known that under reasonable assumptions \cite{benveniste}, \cite{benv_book}, (\ref{eqn1})
 asymptotically tracks the o.d.e.
\begin{align}
\label{eqn2}
\dot{\theta}(t)=h(\theta(t)),
\end{align}
where $h(\theta)= \int f(\theta, y)\Gamma_{\theta}(dy)$, with $\Gamma_{\theta}$ being the unique stationary distribution
of the Markov iterate-dependent process $\{Y_n\}$ for a fixed $\theta$.
Among them the most important assumption is the \textit{stability} of the iterates, i.e.,
\begin{align}
\label{stability}
\sup_n \|\theta_n\| < \infty \ a.s.
\end{align}
In literature sufficient conditions that guarantee (\ref{stability}) are available (e.g. based on a Lyapunov function \cite[Chap. 6.7]{kushner}, \cite{andrieu} etc. As mentioned in \cite{andrieu}, proving stability of the iterates
is a tedious task with the Markovian dynamic due to the noise term $f(\theta_n,Y_{n+1}) - h(\theta_n)$. 
In \cite{andrieu1}, the 
truncations
on adaptive truncation sets from \cite{fu_chen} has been extended to the case where the noise is Markov. 
It is clearly mentioned there that the procedure they follow is 
different
in some respects from the original
procedure proposed by \cite{fu_chen}. To prove that the number of re-initializations of the procedure described in 
\cite[Section 3.2]{andrieu1} is finite, they establish a bound on the probability that the $n$-th reinitialization time is 
finite in terms of the fluctuations of the noise sequence of the algorithm between successive re-initializations. Although the results are described for Markov processes taking values in a general state space (Polish) some of the assumptions therein are restrictive such as the assumptions on global Lyapunov function (see \textbf{(A1)} in \cite{andrieu1}). Additionally, 
in order to control the fluctuations some less classical assumptions have been imposed on the transition kernel (regularity properties in $V$ and $V^p$ norm) as well 
as on the vector field $g(\cdot,\cdot)$ (see (DRI2) and (DRI3)) and the discussion thereafter. 
\begin{comment}
Further, the stability theorem  
stated in \cite[Chap 6, Theorem 9]{borkar_book}) also
%(no proof available in the literature when the noise is controlled Markov) 
requires assumptions such as
1) continuity of the transition kernel, 2) Lipschitz continuity of $f$ in the first component \textit{uniformly} w.r.t 
the second and 3) $f$ being jointly continuous.
\end{comment}

In
this work, we investigate the dynamics of stochastic approximation with  Markov iterate-dependent noise 
when (\ref{stability}) is not known to be satisfied beforehand. We achieve the same by 
extending the  \textit{lock-in probability} framework of Borkar \cite{lock_in_original} to such a recursion, 
leading in turn to the following:
\begin{enumerate}
 \item Let $H$ be an asymptotically stable attractor of (\ref{eqn2}) and $G$ be its domain of attraction. 
If $\{\theta_n\}$ is \textit{asymptotically tight} (which is a much weaker condition than (\ref{stability})) 
and $\liminf_n P(\theta_n \in G) =1,$ then $P(\theta_n \to H) =1$ under a reasonable 
set of assumptions seem to be routinely  satisfied in applications such as reinforcement learning \cite{borkar_meyn}.
%Thus under some restrictive assumptions we prove almost sure convergence of such recursions 
%without assuming stability of the iterates, however, following the classic Poisson equation model of 
%Metivier and Priouret \cite{benveniste} which is designed keeping 
%in mind the stability  of the iterates.
To the best of our knowledge this is the first time 
an almost sure  convergence proof for such a recursion is presented 
without assuming stability of the iterates, however following the classic Poisson equation
model of Metivier and Priouret \cite{benveniste} for such recursion which is designed keeping 
in mind the stability  of the iterates.
%in mind the stability of the iterates. To our knowledge this has not been done before. 
Additionally, a  simple test for asymptotic tightness is also provided. Our tightness condition does not assume existence of a quadratic Lyapunov function as with \cite{sameer} as they work mainly when the vector field is linear. Also, we show that for linear stochastic approximation our tightness condition always gets satisfied in the case of a finite state Markov chain.
%We also show that the first condition as well as other assumptions is easily verifiable compared 
%to many other sufficient conditions for stability and
%convergence for such recursions in literature \cite[p~165, p~192]{kushner}).
 \item We show that under some reasonable assumptions for common step-size sequences such as
$\{\frac{1}{n^k}\}, \frac{1}{2} < k \leq 1$ and $\frac{1}{n (\log n)^{k}}, k\leq 1$, if
the iterates belong to some special open set with compact closure in the domain of attraction of the local attractor
infinitely often w.p. 1, then the iterates are stable and converge a.s. to the local attractor.
 \item We show that our results can be used to analyze the tracking ability of 
general (not necessarily linear) stochastic approximation driven by another ``slowly'' varying 
stochastic approximation process when the iterates are not known to be stable. This involves extending the lock-in probability results to the case with two iterates evolving along different timescales. Such results are useful in the context of 
adaptive algorithms \cite{konda_linear} as not much is known about the stability of frameworks with different 
timescales. 
Note that in two time-scale stochastic approximation the coupled o.d.e has no attractor. Also, to prove the convergence of the coupled iterates stability of the slower iterate is necessary. 
Therefore we need to consider two 
quantities describing difference (over compact time interval) 
between the algorithm and the o.d.e.,  one for the coupled algorithm/o.d.e and
another for the slower algorithm/o.d.e. This gives rise to a situation where the 
conditioning event in the martingale concentration
inequality does not belong to the first $\sigma$-field in the current collection of $\sigma$-fields (unlike in the case of  
single timescale stochastic approximation where the conditioning event always belongs to the 
first $\sigma$-field in the current collection of $\sigma$-fields \cite[p~40]{borkar_book}). 

Such results are useful in the context of 
adaptive algorithms \cite{konda_linear} as not much is known about the stability of frameworks with different 
timescales. There is some recent work \cite{dalal} that also estimates the lock-in probability for 
multiple timescales, however, under the assumption that the vector fields are ``linear''.
\item We derive a sample complexity estimate (explained later) for such a recursion.
\end{enumerate}

The motivation for lock-in probability comes from a phenomenon noticed by W.B.Arthur
in simple urn models (\cite[Chap. 1]{borkar_book}) of increasing return economics:
if occurrences
predominantly of one type tend to fetch more occurrences of the same type, then after
some initial randomness the process gets \textit{locked into}
that possibly undesirable type of occurrence.
Moreover, it is known that
under reasonable conditions, every asymptotically stable equilibrium will have a positive probability
of emerging as $\lim_{n\to \infty}\theta_n$ \cite{arthur},
while this probability is zero for unstable equilibria under mild conditions on the noise \cite{brandiere,permantle}.

With this picture in mind and to give a quantitative explanation of this phenomenon,
Borkar defined lock-in probability \cite{lock_in_original} for  iterates of the form
\begin{align}
\theta_{n+1} = \theta_n + a(n)(h(\theta_n) + M_{n+1}), \label{RM}
\end{align}
where $\{M_n\}$ constitutes martingale noise, i.e., a martingale difference sequence,
as the probability of convergence of $\theta_n$ to an asymptotically stable attractor $H$ of (\ref{eqn2})
 \textit{given} that the iterate is in a neighbourhood $B$ thereof after a \textit{sufficiently large} $n_0$, i.e.,
\begin{align}
P(\theta_n \to H | \theta_{n_0} \in B) \nonumber
\end{align}
for a compact $B \subset G$. He also found a lower bound for this quantity by studying the \textit{local} behavior of iterates in a neighborhood of the attractor.
Clearly, $n_0$ depends on the particular $H$. Specifically, under the assumption
$E[\|M_{n+1}\|^2|\mathcal{F}_n]\leq K(1+\|\theta_n\|^2)\mbox{~a.s.}$ the bound obtained is $1-O(\sum_{i\geq n_0} a(i)^2)$
and under the more restrictive condition  $\|M_{n+1}\| \leq K_0 (1+\|\theta_n\|)\mbox{~a.s.}$, a
tighter bound of $1-O(e^{-\frac{1}{\sum_{i\geq n_0} a(i)^2}})$ has been obtained \cite{lock_in_original}. 
There are recent results \cite{gugan,sameer} which obtain tighter bound under much weaker 
assumptions on martingale and step-size sequence.  

The fact that lock-in probability is not just a theoretical quantity to explain the lock-in phenomenon of information economics
was shown by Kamal \cite{sameer}. If the iterates are \textit{tight} then lock-in probability results are used in
\cite{sameer} to prove almost sure
convergence of the stochastic approximation recursion (with only martingale difference noise) to the \textit{global} attractor.
%bypassing the stability proof which may be difficult at times.

The phenomenon described earlier can be observed in reinforcement learning (RL) applications  where the
limiting o.d.e.\ has multiple equilibria, e.g., with several instances of stochastic gradient descent in machine learning.
%As many RL algorithms contain non-additive Markov iterate-dependent noise in the iterates, we need to extend the present lock-in probability estimates to the
%case where the vector field includes a Markov iterate-dependent noise.
We extend in this paper the currently available lock-in probability estimates to the case where the vector field includes a Markov iterate-dependent
noise. This is for instance the case with many reinforcement learning algorithms where Markov iterate-dependent noise arises naturally because of the Markov decision process in the background.

%Then the following algorithm will serve the purpose in all the cases:
%\begin{align}
%\label{main}
%\theta_{n+1} = \theta_n + a(n)f(\theta_n,Y_{n+1}).
%\end{align}
%In the first case one can put it into the framework of (\ref{eqn1}) whereas in the second case we need to follow the approach of
%stochastic approximation with  Markov iterate-dependent noise \cite{benveniste, borkar_} to analyze its convergence properties.

Although the recursion (\ref{eqn1}) covers most of the cases of stochastic approximation with  Markov iterate-dependent noise,
there are reinforcement learning scenarios where there can be a dependence on both the present and the next sample of the Markov iterate-dependent noise
in the vector field \cite{off-policy}. For such scenarios the general recursion is:
\begin{align}
\label{many}
\theta_{n+1} = \theta_n + a(n)f(\theta_n,Y_{n},Y_{n+1}).
\end{align}
One can write (\ref{many}) as
\begin{align}
\theta_{n+1} = \theta_n + a(n)\left[E[f(\theta_n,Y_{n},Y_{n+1})|\mathcal{F}_n] + M_{n+1}\right],\nonumber
\end{align}
where $\mathcal{F}_n=\sigma(\theta_m,Y_m,m\leq n)$ and $M_{n+1}=f(\theta_n,Y_{n},Y_{n+1})- E[f(\theta_n,Y_{n},Y_{n+1})|\mathcal{F}_n]$ is a martingale difference sequence.
Therefore, with abuse of notation, the general recursion which takes care of 
Markov iterate-dependent noise can be
described as
\begin{align}
\label{main_m}
\theta_{n+1} = \theta_n + a(n)\left[f(\theta_n,Y_{n}) + M_{n+1}\right].
\end{align}
In fact, this also covers the situation where both Markov iterate-dependent \textit{and} martingale difference noise sequences are present.
In this work, we give a lower bound on the lock-in probability estimate
%for the stochastic approximation iterates with  only Markov iterate-dependent
of iterates of the form (\ref{main_m}) using the \textit{Poisson equation} based analysis
as in \cite{benveniste, benv_book}. 
%Here we additionally
%assume that the norm of the Markov iterate-dependent process is upper bounded by a constant almost surely.
 %\item Due to the lack of the stability of the iterates
%we assume that the transition kernel does not depend on the iterates i.e. the underlying process is a un Markov iterate-dependent process
%with non-unique stationary distributions (\cite{borkar_}).
Under some assumptions in \cite{benveniste} and some further assumptions,
we get a lower bound of $1-O(e^{-\frac{C}{\sum_{i=n_0}^{\infty}a(i)^2}})$ for the recursion (\ref{main_m}),
and thus also for the special case (\ref{eqn1}). Therefore, with the more general assumption of Markov iterate-dependent noise, we recover
the \textit{same bounds} available for the setting of martingale noise \cite[p.~38]{borkar_book} although with
some additional assumptions on the Markov iterate-dependent process and step size sequence.

Very few results \cite{gen} are 
available on non-asymptotic rate of
convergence of general stochastic approximation iterates (\ref{eqn1}), see also \cite{rakhlin} for stochastic gradient descent.
But lock-in probability estimates can be used to calculate an upper bound for the sample complexity
estimate of stochastic approximation \cite[chap. 4.2]{borkar_book},\cite{sameer}.
% lock-in probability results can be used to
%derive a sample complexity estimate, i.e.
Given a desired accuracy $\epsilon >0$ and confidence $\gamma$, the 
sample complexity estimate is defined to be the minimum number of iterations
$N(\epsilon, \gamma)$ after which the iterates are within a certain neighbourhood
(which is a function of $\epsilon$) of $H$ with probability at least $1-\gamma$.
This is slightly different from the sample complexity estimate
arising in the context of
consistent supervised learning algorithms in statistical learning theory \cite{lorenzo}. The reasons are:
\begin{enumerate}
 \item In the case of statistical learning theory, sample complexity corresponds to
the number of \textit{i.i.d training samples} needed for the algorithm to successfully learn a target function. However,
in our case, we have a \textit{recursive} scheme whose sample complexity depends on the step-size.
%\item Moreover, the estimate is of an
%asymptotic nature, which requires this $n_0$ to be \textit{large enough} (so that the
%decreasing step size  and the tail sum of the squares of the step-sizes has decreased sufficiently).
\item Ours is a \textit{conditional} estimate, i.e.,
 the estimate is conditioned on the fact that $\theta_{n_0} \in B$ where $B$ is
an open subset of the domain of attraction of $H\subset B$ and has compact closure, and $n_0$ is sufficiently large.
\end{enumerate}
Another point worth noting is that sample complexity results are much weaker than lock-in probability and do not
require existence of Lyapunov function. In our work, we give a sample complexity estimate
for the setting where the recursion is a \textit{stochastic fixed point} iteration driven by a Markov iterate-dependent noise. This
shows a \textit{quantitative} estimate of \textit{large vs. small step size trade-off} well known in stochastic
approximation literature that is shown to be useful in choosing the \textit{optimal step-size}.
%\end{comment}

The organization of the paper is as follows: Section \ref{sec_def} formally
defines the problem and provides background and assumptions.
Section \ref{main_res} presents our \textit{lock-in probability}
results for single timescale stochastic approximation. 
%Section \ref{discuss} discusses the adaptability of our proof technique over some alternate set of assumptions.
Section \ref{a.s.conv} presents results on \textit{almost sure convergence} to a
\textit{local} attractor using our results along with the assumption of  \textit{asymptotic tightness}
of the iterates. Moreover, this section shows that stability of the iterates can be proved using our results.
Section \ref{track} gives an estimate of the lock-in probability for iterates evolving along different timescales and analyzes the tracking ability of adaptive algorithms using these results.
Section \ref{sample} describes the results on sample complexity.
%gives a sufficient condition for almost sure convergence
%in absence of asymptotic tightness.
Finally, we conclude by providing
some future research directions.
\section{The Problem and Assumptions}
\label{sec_def}

In the following we describe the preliminaries and notation that we use in our proofs.
Most of the definitions and notation are from 
\cite{benveniste, borkar_book, sameer}. The notations used 
for ordinary differential equations are similar to \cite[Appendix 11.2]{borkar_book}.
In the following we describe the lock-in probability settings based on the approach in \cite{benveniste}.
The main idea is to assume existence of a solution to the Poisson equation (Assumption (M4) from Section III B of \cite{benveniste}),
thus converting Markov iterate-dependent noise into a martingale difference sequence and with additional additive errors. We refer the readers
to \cite[Part II,Chap. 2, Theorem 6]{benv_book}, \cite[Section III~D, Appendix A]{benveniste} for details on
the existence and properties of solution of the Poisson equation for a Markov iterate-dependent process.

In this work we prove 
almost sure  convergence for recursion (\ref{main_m}) 
without assuming stability of the iterates, however, following the classic Poisson equation
model stated above where the assumptions are designed keeping in mind the stability of the iterates. 
To make up for this we need to strengthen 
some existing assumptions of \cite{benveniste} (shown next), these are standard 
assumptions satisfied in application areas such as reinforcement learning.
%Under the assumptions made in \cite{benveniste} (which include the stability of the iterates), it is well-known that the limiting
%o.d.e for (\ref{eqn1}) is
%\begin{align}
%\label{ode}
%\dot{\theta}(t)=h(\theta(t))
%\end{align}
%where $h(\theta) =  \int f(\theta,y)\Gamma_{\theta}(dy)$.

Let $G \subset \mathbb{R}^d$ be open and let $V:G\to [0,\infty)$ be such that $\langle \nabla V, h \rangle : G \to \mathbb{R}$ is non-positive.
We shall assume as in \cite{borkar_book}  that $H:=\{\theta:V(\theta) = 0\}$ is equal to the set $\{\theta: \langle \nabla V(\theta), h(\theta)\rangle=0\}$
and is a compact subset of $G$.
Thus, $V$ is a strict Lyapunov function. Then $H$ is an asymptotically stable invariant set of the
differential equation $\dot{\theta}(t) = h(\theta(t))$.
Let there be an open set $B$ with compact closure such that
$H \subset B \subset \bar{B} \subset G$. In this setting, the lock-in probability is defined to be the probability that the sequence
$\{\theta_n\}$ is convergent to $H$, conditioned on the event that $\theta_{n_0}\in B$ for some $n_0$ sufficiently large.

Recall from  Theorem 8 of
\cite[p.~37]{borkar_book} that for the case of martingale difference noise, $\mathbb P[\theta_n\rightarrow
H|\theta_{n_0}\in B] \geq 1 - O(e^{-\frac{1}{s(n_0)}}),$ where $s(n_0) :=
\sum_{m=n_0}^{\infty} a(m)^2$.
%and $B$ is a bounded open set
%contained in the domain of attraction of $H$.
In this paper we obtain these results when the noise is  Markov iterate-dependent under the following assumptions:
%than \textbf{(M1)}, \textbf{(M4)}, \textbf{(M5 b,c)} mentioned in \cite[Section ~III]{benveniste}:
\begin{enumerate}[label=\textbf{(A\arabic*)}]
 \item $\limsup_{n\to \infty}\|Y_n\|< \bar{C}$ a.s. for some $\bar{C} > 0$.
This is stronger than  $\limsup_n E[\|Y_n\|^2] < \infty$ which is implied by 
\textbf{(M2)} of \cite{benveniste}.
\item $\sup_{y}\|f(\theta,y)\| \leq K(1+\|\theta\|)$ for all $\theta$.
\begin{rmk}
\textbf{(A2)} is a standard assumption satisfied in reinforcement learning scenarios as 
pointed in \cite[p~6]{lock_in_original}. Clearly, this is stronger than 
the hypothesis (F) on $f$ as mentioned in \cite[p~143]{benveniste}.
\end{rmk}  
\item  The stepsizes $\{a(n)\}$ are non-increasing positive scalars satisfying
\begin{align}
\sum_n a(n) = \infty, \sum_{n}{a(n)}^2 < \infty.\nonumber
\end{align}

\item For every $\theta$, the Markov chain $\Pi_\theta$ has a unique invariant probability $\Gamma_\theta$. (\textbf{(M1)} from 
\cite{benveniste}). 
%Let $\Gamma_{\theta}$ be the unique stationary distribution of the Markov iterate-dependent process for a
%fixed $\theta$ (\cite{benveniste}).
Further, $h(\theta) = \int f(\theta,y)\Gamma_{\theta}(dy)$ is Lipschitz continuous in $\theta$ with Lipschitz constant
$0<L<\infty$.
\item $\|M_{n+1}\| \leq K'(1+\|\theta_n\|)$ a.s. $\forall n$. Note that in \cite{benveniste} 
there was no martingale noise.
\item For every $\theta$ the Poisson equation
$$(1 - \Pi_{\theta})v_\theta = f(\theta,\cdot)-\int f(\theta,y)\Gamma_\theta(dy)$$
has a solution $v_\theta$. This is \textbf{(M4)} from \cite{benveniste}. Here $\Pi_\theta \phi(x):=\int \Pi_\theta(x;dy) \phi(y)$.
\item For all $R>0$ there exist constants $C_R>0$ such
that
\begin{enumerate}
\item $\sup_{\|\theta\|\leq R} \|v_{\theta}(x)\| \leq C_R (1+\|x\|)$.
\item $\|v_{\theta}(x)-v_{\theta'}(x)\| \leq C_R \|\theta - \theta'\|(1+ \|x\|)$ for all $\|\theta\| \leq R$, $\|\theta'\| \leq R$.
\end{enumerate}
This is \textbf{(M5)b,c} from \cite{benveniste}.
\end{enumerate}

Under the above assumptions we shall show that
\[
 \mathbb P\left[\theta_n\rightarrow H|\theta_{n_0}\in B\right] \geq 1 - O\left(e^{-\frac{c}{s(n_0)}}\right)
\]
 for sufficiently large $n_0$.

We provide a more detailed discussion on assumptions \textbf{(A1)} and \textbf{(A2)} as well as 
possible relaxations of these in Section \ref{discuss}. Additionally, we give a detailed comparison of our assumptions with the assumptions of \cite{andrieu1} in Section \ref{comp}.
%Other than the assumptions made in Section III B of \cite{benveniste} on $Y_n, n\geq 0$we make the following new assumptions:
%\begin{rmk}
%When $Y_n, n\geq 0$  are  i.i.d in (\ref{eqn1}), assumption (A2) will be required
%to satisfy the assumption on martingale difference noise, namely,
%%$\|M_{n+1}\|\leq K(1+ \|\theta_n\|)$
%\textbf{(A5)} which helps to get an exponential bound. There are
%some results which derive the same bound under the much weaker assumption like
%$P(\|M_{n+1}\|> u|\mathcal{F}_n) \leq c_1 e^{-c_2 \frac{u}{1+\|\theta_n\|}}$ for sufficiently large $u$.
%However they are much harder to verify (\cite{sameer}) as exponential tail bounds
%are available in literature only as unconditional probability
%for martingales with almost surely bounded differences \cite{mcdarmiad}.
%%Another reason is that for any random variable $X$ with $|X| < M$ a.s. and sigma field $\mathcal{F}$
%%\begin{align}
%%P(|X| > a | \mathcal{F}) \leq \exp(- \frac{\frac{1}{2}t^2}{E[X^2|
%%\end{align}
%Therefore (A1) and (A3 ii) are the only two new assumptions
%we make compared to the case of i.i.d noise to recover the same bound.
%\end{rmk}
%\begin{rmk}
%In the convergence proof of \cite{benveniste}, stability of the iterates and \linebreak
%$\sup_k E[\|Y_k\|^2]<\infty$ were assumed.
%Therefore the weighted (by step-size) sum of additional errors (arising after using Poisson equation) over a
%$T$-length interval were made to converge to zero by Borel-Cantelli lemma I. Here we need \textbf{(A1)} and \textbf{(A2)}
%because stability of the iterates is not assumed.
%\end{rmk}
\section{Lock-in probability calculation for single timescale stochastic approximation}
\label{main_res}

In this subsection we give a lower bound for $\mathbb P[\theta_n\rightarrow
H|\theta_{n_0}\in B]$ in terms of $s(n_0)$ when $n_0$ is sufficiently
large based on the settings described in Section \ref{sec_def}. How large $n_0$ needs to be will be specified soon.
Before proceeding further we describe our notations
and recall some known results. For $\delta > 0$, $N_{\delta}(A)$ for a set $A$ denotes its $\delta$-neighborhood $\{y : \|y - x\| < \delta, \forall \ x \in A\}$.
Let $H^a:=N_{a}(H)$. Fix some $0 < \epsilon_1 < \epsilon$
and $\delta_B>0$ such that $N_{\delta_B}(H^{\epsilon_1})\subset H^{\epsilon} \subset B$.

Let
%\begin{align}
%T = \frac{[\max_{\theta \in \bar{B}} V(\theta)]-\epsilon_1}{\min_{\theta \in \bar{B} \setminus H^{\epsilon_1}}|\langle\nabla V(\theta),h(\theta)\rangle|}.
%\end{align} Then 
$T$ be an upper bound 
for the time required for a solution of the o.d.e.\ (\ref{eqn2}) to reach the set
$H^{\epsilon_1}$, starting from an initial condition in $\bar{B}$. 
The existence of such a $T$ independent of the starting point in $\bar{B}$ can be proved 
using the continuity of flow of the o.d.e (\ref{eqn2}) around the initial point and the fact that $H$ is an
asymptotically stable set of the same o.d.e; see Lemma 1  of \cite[Chap. 3]{borkar_book} for the proof of a similar result.
\begin{rmk}
 Due to the above argument we do not need assumption \textbf{(A4)} from \cite{gugan} and similar assumption made in \cite[p~32]{borkar_book}. They clearly gets satisfied if the vector field is linear, however for nonlinear stochastic approximation it is not clear whether they gets satisfied.  
\end{rmk}

 Let $t(n)= \sum_{m=0}^{n
- 1}a(m)$, $n \geq 1$ with $t(0)=0$. Let $n_0 \geq 0, n_m = \min \{n:t(n)\geq t(n_{m-1})+T\}$ and $T_m = t(n_m)$, $m \geq 1$.
Define $\bar \theta(t)$ by: $\bar{\theta}(t(n)) = \theta_n$, with linear
interpolation on $[t(n), t(n+1))$ for all $n$.
 Let $\theta^{t(n_m)}(\cdot)$ be the solution of the limiting o.d.e.\ (\ref{eqn2}) on $[t(n_m), t(n_{m+1}))$ with
 the initial condition $\theta^{t(n_m)}(t(n_m))=\bar \theta (t(n_m)) =  \theta_{n_m}$.
  Let
 \[
 \rho_m := \sup_{t \in [t(n_m), t(n_{m+1}))} \|\bar \theta(t) - \theta^{t(n_m)}(t)\|.
 \]

 We recall here a few key results from \cite{lock_in_original}. As shown there, 
if $\theta_{n_0}\in B$, and $\rho_m < \delta_B$ for all $m \geq 0$, then $\bar{\theta}(T_n)$ is in $H^{\epsilon} \subset B$
for all $n \geq 1$. Therefore using discrete Gronwall's inequality we can show that $\sup_{t \geq T_0} \bar{\theta}(t) < \infty$.
It is also known (\cite{benveniste}, section IIC) that if the sequence of iterates
 $\{\theta_n\}$ remains bounded almost surely on a prescribed set of
 sample points, and if on this set the iterates lie in a compact set in the domain of 
attraction of any local attractor infinitely often then the sequence of iterates converge almost surely on this
 set to that local attractor.
%(possibly sample path dependent) compact internally chain transitive invariant set of o.d.e (\ref{eqn2}). Additionally,
%it follows from the LaSalle invariance principle that any internally chain transitive invariant subsets of $\bar{B}$ will be
%subsets of $H$.
Using this fact gives  the following estimate on the
 probability of convergence, conditioned on $\theta_{n_0}\in B$ (\cite{borkar_book}, Lemma 1, p.\ 33):
 \[
 P\left[\bar \theta(t)\rightarrow H | \theta_{n_0}\in B\right] \geq
 P\left[\rho_m < \delta_B \ \forall m \geq 0 | \theta_{n_0}\in B
 \right].
 \]
Let $\mathcal{B}_m$ denote the event that $\theta_{n_0}\in B$ and
 $\rho_k < \delta_B$ for $k=0,1,\ldots,m$. Clearly, $B_m \in \mathcal{F}_{n_{m+1}}$. The following lower bound
 for the above probability has been obtained in (\cite{borkar_book}, Lemma 2, p.\ 33):
 \[
 P\left[\rho_m < \delta_B \ \forall m \geq 0 | \theta_{n_0}\in B
 \right]\geq 1-\sum_{m=0}^\infty P\left[\rho_m \geq \delta_B |\mathcal{B}_{m-1} \right].
 \]

Subsequently the idea is to find an upper bound for $\rho_m$ consisting of errors (asymptotically negligible on $\mathcal{B}_{m-1}$) 
as well as martingale terms. Then for some large $n_0$, one may bound 
$P(\rho_m \geq \delta_B|\mathcal{B}_{m-1})$ using a suitable martingale 
concentration inequality. In the following we describe how to achieve the above in our setting. 

Using the Poisson equation
one can write the recursion (\ref{main_m}) as
\small
\begin{align}
\theta_{n+1}= \theta_{n}+ a(n)h(\theta_n) + a(n)\left[v_{\theta_n}(Y_{n})-\Pi_{\theta_n}v_{\theta_n}(Y_{n})+M_{n+1}\right]\nonumber
\end{align}
\normalsize
where $\Pi_{\theta} \phi(x) = \int \phi(y)\Pi_{\theta}(x;dy)$.
Let $\zeta_{n+1} = v_{\theta_n}(Y_{n}) -\Pi_{\theta_n}v_{\theta_n}(Y_{n})$.
We decompose
\begin{align}
\begin{aligned}
\zeta_{n+1} = v_{\theta_n}(Y_{n+1}) -\Pi_{\theta_n}v_{\theta_n}(Y_{n}) + v_{\theta_n}(Y_{n}) - v_{\theta_{n+1}}(Y_{n+1})+  \\ v_{\theta_{n+1}}(Y_{n+1})- v_{\theta_{n}}(Y_{n+1})\nonumber
\end{aligned}
\end{align}
and set
\begin{gather*}
A_n = \sum_{k=0}^{n-1}a(k)\zeta^{(1)}_{k+1}, \ B_n = \sum_{k=0}^{n-1}a(k)\zeta^{(2)}_{k+1}, \ C_n = \sum_{k=0}^{n-1}a(k)\zeta^{(3)}_{k+1}, \nonumber \\
D_n = \sum_{k=0}^{n-1}a(k)M_{k+1}, n \geq 1\nonumber
\end{gather*}
where
\begin{gather*}
\zeta^{(1)}_{n+1} =  v_{\theta_n}(Y_{n+1}) -\Pi_{\theta_n}v_{\theta_n}(Y_{n}), \ \zeta^{(2)}_{n+1} = v_{\theta_n}(Y_{n}) - v_{\theta_{n+1}}(Y_{n+1}),\nonumber \\
\zeta^{(3)}_{n+1} = v_{\theta_{n+1}}(Y_{n+1})- v_{\theta_{n}}(Y_{n+1}).
\end{gather*}

Then one can easily see that as in the proof of Lemma 3 of \cite[p.~34]{borkar_book}
\small
\begin{align}
\label{rho}
\rho_m \leq & (C a(n_0) + K_T CLs(n_0)) + K_T [ \max_{n_m \leq j \leq n_{m+1}} \|A_j - A_{n_m}\| + \nonumber \\
            & \max_{n_m \leq j \leq n_{m+1}} \|B_j - B_{n_m}\|+  \max_{n_m \leq j \leq n_{m+1}} \|C_j - C_{n_m}\| + \nonumber \\ 
            & \max_{n_m \leq j \leq n_{m+1}} \|D_j - D_{n_m}\|],
\end{align}
%\end{align}
\normalsize
where $C$ is a bound on $\|h(\Phi_t(\theta)\|$, with $\Phi_t$ being the time-$t$ flow map for the o.d.e
(\ref{eqn2}), $0\leq t \leq T+1$ and $\theta \in \bar{B}$. Also, $K_T=e^{LT}$.

Choose an $n_0^{(1)}$ such that
\begin{align}
\label{1}
(C a(n_0^{(1)}) + K_T CLs(n_0^{(1)})) < \delta_B/2.
\end{align}

%Before proving some useful lemmas we recall a few important properties involving solution of Poisson
%equation $v: \mathbb{R}^d \times \mathbb{R}^k \to \mathbb{R}^d$ from \cite{benveniste} which will be used
%frequently.

%For all $R > 0$ there exist constant $C_R>0$ such that
%\begin{enumerate}[label=\textbf{(P\arabic*)}]
%\item \begin{align}
% \|v_{\theta}(x)-v_{\theta'}(x)\| \leq C_R \|\theta - \theta'\|(1+ \|x\|)\nonumber
%\end{align}
%for $\|\theta\| \leq R$, $\|\theta'\| \leq R$.
%\item $\sup_{\|\theta\|\leq R} \|v_{\theta}(x)\| \leq C_R (1+\|x\|)$.
%\end{enumerate}

%\begin{lem}
%\label{benv_lem}
%For each $R$ there exist constant $C_R'$ such that
%\begin{enumerate}
% \item $\|\Pi_\theta v_\theta(x) - \Pi_{\theta'} v_{\theta'}(x)\| \leq C_R'(1+\|x\|)(\|\theta - \theta'\|)$ for every $\theta, \theta'$
%with $\|\theta\| < R$, $\|\theta'\| < R$
%\item $\sup_{\|\theta\| \leq R}\|\Pi_\theta v_\theta(x)\|\leq C_R'(1+\|x\|)$
%\end{enumerate}

%\end{lem}
%After this we prove this useful lemma:
The following important lemma shows that $\forall m \geq 1$, on $\mathcal{B}_{m-1}$, the iterates (\ref{main_m}) are stable over $T$-length interval with the stability 
constant independent of $m$. This is enough for our proofs to go through and justifies the importance of assumptions 
\textbf{(A2)} and \textbf{(A5)}. 
\begin{lem}
\label{T_stability}
On $\mathcal{B}_{m-1}, \|\theta_j\| \leq K''$ for any $n_m \leq j \leq n_{m+1}$ where the constant $K''$ is independent of $m$.
\end{lem}
\begin{proof}
From the definition of $\mathcal{B}_{m-1}$, we know that $\theta_{n_m} \in B$ on this event. Let $\|\theta_{n_m}\| \leq \tilde{C}$ $\forall m$.
Clearly, for $n_m \leq j \leq n_{m+1}$,
\begin{align}
 \begin{aligned}
  \|\theta_j\| &\leq \|\theta_{n_m}\| + \sum_{k=n_m}^{j-1} a(k)\left[\|f(\theta_k, Y_{k})\|+\|M_{k+1}\|\right]\nonumber \\
               &\leq \tilde{C} + \tilde{K} \sum_{k=n_m}^{j-1} a(k) (1+ \|\theta_k\|)
 \end{aligned}
\end{align}
where $\tilde{K} =\max(K,K')$.
As $\sum_{k=n_m}^{j-1} a(k) \leq T$, discrete Gronwall inequality gives the result.
\end{proof}

\begin{lem}
For sufficiently large $n_m$,  $\max_{n_m \leq j \leq n_{m+1}} \|B_j - B_{n_m}\| < \frac{\delta_B}{8K_T}$ a.s.
on the event $\mathcal{B}_{m-1}$.
\end{lem}
\begin{proof}
Now, if we write
$B_{n_m} = a(0)v_{\theta_0}(Y_0) + \sum_{k=1}^{n_m-1}(a(k) - a(k-1))v_{\theta_k}(Y_k) - a(n_m-1) v_{\theta_{n_m}}(Y_{n_m})$
we obtain
\begin{align}
B_j - B_{n_m} = &\sum_{k=n_m}^{j-1}(a(k) - a(k-1))v_{\theta_k}(Y_k) +\nonumber \\
 &a(n_m-1) v_{\theta_{n_m}}(Y_{n_m})-a(j-1) v_{\theta_j}(Y_j).\nonumber
\end{align}
As $\|\theta_i\| \leq K'$ on $\mathcal{B}_{m-1}$ 
\begin{align}
\|B_j - B_{n_m}\| \leq C_R \sum_{k=n_m}^{j-1}(a(k-1) - a(k))(1+\|Y_k\|) +  \nonumber \\
C_R\left[ a(n_m-1)(1+ \|Y_{n_m}\|) +  a(j-1)(1+ \|Y_{j}\|)\right]\nonumber
\end{align}
using \textbf{(A7a)}.
Now using \textbf{(A1)}, \textbf{(A3)}\footnote{This is the only place where the requirement that step size is non-increasing in \textbf{(A3)} is used.}
we see that
\begin{align}
\|B_j - B_{n_m}\| \leq 2C_R''a(n_m-1),\nonumber
\end{align}
for some $C_R'' >0$.
Now choose $n_0^{(2)}$ such that
\begin{align}
\label{2}
2C_R''a(n_0^{(2)}-1) < \frac{\delta_B}{8K_T}.
\end{align}
The claim follows $\forall n_m \geq n^{(2)}_0$.
\end{proof}

\begin{lem}
For sufficiently large $n_m$,  $\max_{n_m \leq j \leq n_{m+1}} \|C_j - C_{n_m}\| < \frac{\delta_B}{8K_T}$ a.s. 
on the event $\mathcal{B}_{m-1}$.
\end{lem}
\begin{proof}
Using \textbf{(A7b)} we see that
\begin{align}
\|\zeta^{(3)}_{k+1}\| \leq C_R\|\theta_{k+1} - \theta_k\|(1+\|Y_{k+1}\|).\nonumber
\end{align}
Again using the stability of the iterates in the $T$ length interval on $\mathcal{B}_{m-1}$ and the assumptions \textbf{(A1)} and
\textbf{(A2)}
we see that \begin{align}
\|\zeta^{(3)}_{k+1}\| \leq C_R\tilde{K}\bar{C}a(k).\nonumber
\end{align}
Therefore \begin{align}
\|C_j - C_{n_m}\| \leq C_R\tilde{K}\bar{C} \sum_{k=n_m}^{j-1} a(k)^2. \nonumber
\end{align}
Now choose $n_0^{(3)}$ such that
\begin{align}
\label{3}
C_R \tilde{K}\bar{C} \sum_{k=n_0^{(3)}}^{j-1} a(k)^2 < \frac{\delta_B}{8K_T}.
\end{align} This is possible
due to  \textbf{(A3)}. The claim follows for $n_m \geq n^{(3)}_0$.
\end{proof}
\begin{thm}
\label{main_thm}
For $n_0$ sufficiently large,
\begin{align}
P(\bar{\theta}(t) \to H| \theta_{n_0} \in B) \geq 1-2de^{-\frac{\hat{K}\delta_B^2}{ds(n_0)}}- 2d e^{-\frac{\hat{C}\delta_B^2}{ds(n_0)}}.\nonumber
\end{align}
\end{thm}
\begin{proof}
Set
\begin{align}
\label{large}
n_0 = \max (n_0^{(1)}, n_0^{(2)}, n_0^{(3)}).
\end{align}.
From (\ref{rho}) we see that
for this (large) $n_0$
%\begin{align}
\small
\begin{align}
P(\rho_m \geq \delta_B |& \mathcal{B}_{m-1}) \leq P(\max_{n_m \leq j \leq n_{m+1}} \|A_j - A_{n_m}\|>\frac{\delta_B}{8K_T}|\mathcal{B}_{m-1})  \nonumber \\
                        & + P(\max_{n_m \leq j \leq n_{m+1}} \|D_j - D_{n_m}\|>\frac{\delta_B}{8K_T}|\mathcal{B}_{m-1}). \nonumber
\end{align}
\normalsize
%\end{align}
Again using the stability of the iterates in the $T$ length interval on $\mathcal{B}_{m-1}$ and assumption
\textbf{(A7a)} we see that $\zeta^{(1)}_{k+1}$ is bounded a.s. on $\mathcal{B}_{m-1}$ by
the constant $C_0=2C_R(1+\bar{C})$ for $n_m \leq k \leq j-1$. Therefore each of the components in
this vector is also bounded by the same constant.
Therefore,
\small
\begin{align}
\nonumber
P(&\max_{n_m \leq j \leq n_{m+1}}  \|A_j - A_{n_m}\|>\delta_B/8K_T|\mathcal{B}_{m-1}) \nonumber \\
&\leq P(\max_{n_m \leq j \leq n_{m+1}} \|A_j-A_{n_m}\|_{\infty} > \frac{\delta_B}{8K_T\sqrt{d}}|\mathcal{B}_{m-1})\nonumber \\
&=P(\max_{n_m \leq j \leq n_{m+1}} \max_{1\leq i \leq d} |A^i_j - A^i_{n_m}| > \frac{\delta_B}{8K_T\sqrt{d}} |\mathcal{B}_{m-1})\nonumber \\
&=P(\max_{1\leq i \leq d} \max_{n_m \leq j \leq n_{m+1}}  |A^i_j - A^i_{n_m}| > \frac{\delta_B}{8K_T\sqrt{d}} |\mathcal{B}_{m-1})\nonumber \\
&\leq \sum_{i=1}^{d}P(\max_{n_m \leq j \leq n_{m+1}}  |A^i_j - A^i_{n_m}| > \frac{\delta_B}{8K_T\sqrt{d}} |\mathcal{B}_{m-1})\nonumber \\
&\leq \sum_{i=1}^{d}2\exp\{-\frac{\delta_B^2}{32K_T^2dC_0^2(\sum_{j=n_m}^{n_{m+1}}a(j)^2)}\}\nonumber \\
&\leq 2d \exp\{-\frac{\delta_B^2}{32K_T^2dC_0^2(\sum_{j=n_m}^{n_{m+1}}a(j)^2)}\}\nonumber \\
&= 2d \exp\{-\frac{\delta_B^2}{32K_T^2dC_0^2[s(n_m)-s(n_{m+1})]}\}\nonumber
\end{align}
\normalsize
In the third inequality above we use the conditional version of the martingale concentration inequality \cite[p.~39, chap.~4]{borkar_book}.
We give an outline of its proof in the Appendix.
Now it can be shown as in Theorem 11 of \cite[Chap. 4]{borkar_book} that for sufficiently large $n_0$,
\begin{align}
P(\rho_m < \delta_B,~\forall m\geq 0| \theta_{n_0} \in B) \geq 1-2de^{-\frac{\hat{K}\delta_B^2}{ds(n_0)}} - 2d e^{-\frac{\hat{C}\delta_B^2}{ds(n_0)}}\nonumber
\end{align}
where $\hat{K} = 1/32K_T^2C_0^2$ and $\hat{C}$ is the same as in Theorem 11 \cite[p~40]{borkar_book}.
\end{proof}
%\begin{cor}
%Under the existence of the pseudo-Lyapunov function as mentioned in \cite[Chap. 4.2]{borkar_book},
%given a desired accuracy $\epsilon >0$ and confidence $\gamma$, one needs
%\begin{align}
%N_0:=\min \left{ n : \sum_{i=n_0 +1}^n} a(i) \geq \frac{(V(\theta_{n_0})-\epsilon)(T+1)}{\Delta/2}
%\end{align}
%more iterates to get within $\delta_{\epsilon} + \epsilon +  \Delta/2$ where $n_0$ should satisfy (\ref{n_0}) and
\subsection{Discussion on the assumptions}
\label{discuss}
\subsubsection{$Y_n$ unbounded}
Even if $Y_n, n\geq 0$ are unbounded and iterate-dependent our analysis will go through in the following case by creating 
functional dependency between $\{Y_n\}$ and $\{\theta_n\}$.
\begin{enumerate}[label=\textbf{(A\arabic*)'}]
\item For large $n$, $\ \lVert Y_{n+1} \rVert \le K_0 (1 + \lVert \theta_n \rVert) \text{ for some $0 < K_0 < \infty$}$.
\end{enumerate}
Such an assumption will be satisfied if $Y_{n+1} = \psi(\theta_n, Y_n)$  with $\psi$ roughly growing 
linearly as a function of $\theta$ alone \textit{i.e,}
$\lVert \psi(\theta, y) \rVert \le K_0 (1 + \lVert \theta \rVert)$. In other words, $\psi$ is point-wise bounded with respect to
$\theta$ alone. 

%Given that $\forall n$ $\theta_n \in \mathbb{R}^d$, $Y_{n} \in \mathbb{R}^k$ and $f: \mathbb{R}^{d+k} \to \mathbb{R}^d$,
%the point-wise boundedness of $f$ implies $\lVert f(z) \rVert \le K(1 + \lVert z \rVert)$
%for all $z \in \mathbb{R}^{d+k}$. In other words, for any $\theta \in \mathbb{R}^d$ 
%and $y \in \mathbb{R}^k$ we have,
%\begin{align} \label{eq:pwb}
% \lVert f(\theta , y) \rVert \le K \left(1 + \lVert (\theta , y) \rVert \right) \le 
%K(1 + \lVert \theta \rVert + \lVert y \rVert ).
%\end{align}
%If the Markov iterate-dependent process evolves satisfying \textbf{(A1)} then 
%we get the following inequality at stage $n$:
%\begin{align} \label{eq:pwbn}
% \lVert f(\theta_n , Y_{n+1}) \rVert \le K(1 + \lVert \theta_n \rVert + \lVert Y_{n+1} \rVert )
% \le K(1 + \lVert \theta_n \rVert + \overline{C}).
%\end{align}
%Without loss of generality $K$ is such that 
%$\lVert f(\theta_n , Y_{n+1}) \rVert \le K(1 + \lVert \theta_n \rVert)$ $\forall n$.
%This is essentially what is used in all our proofs. 

Accordingly we may replace \textbf{(A2)} by
the point-wise boundedness of $f$ \textit{i.e.,}
\begin{enumerate}[label=\textbf{(A\arabic*)'}]
\setcounter{enumi}{1}
\item $\ \|f(\theta, y)\| \le K(1 + \lVert \theta \rVert + \lVert y \rVert)$.
\end{enumerate}
%On the other hand if the evolution of the Markov iterate-dependent process is such that $\textbf{(A1)}'$ is satisfied, 
%, then the last term of (\ref{eq:pwbn}) can be replaced by
%$K(1 + \lVert \theta_n \rVert + K_0 (1 + \lVert \theta_n \rVert))$. Again, without loss of
%generality we have,
%$\lVert f(\theta_n , Y_{n+1}) \rVert \le K(1 + \lVert \theta_n \rVert)$ $\forall n$.
%Even in this case we may replace \textbf{(A2)} by \textbf{(A2)'}.

%Consider the ``'' version of (\ref{eq:un}) given by
%\begin{align}
% Y_{n+1} = A(\theta_n) \ Y_{n}, \ n \ge 1.
%\end{align}
%Assume 
%\begin{enumerate}
% \item $Y_n$ belongs to a compact set $K$ i.o.
% \item $\forall~R>0$, $\sup_{\|\theta\| < R} \|A(\theta)\|<\alpha_R<1$
%\end{enumerate}
%In the above case neither \textbf{(A1)'} nor $Y_n$ bounded is satisfied, still our proofs will go through.
\subsubsection{$Y_n$ pointwise bounded}
Our analysis will also go through (with the addition of an error term) 
for the following relaxation of \textbf{(A1)}:
\begin{enumerate}[label=\textbf{(A\arabic*)''}]
\item $\limsup_n \|Y_n\| < \infty \mbox{~~a.s.}.$ 
\end{enumerate}
In this case the lock-in probability statement in Theorem \ref{main_thm} will be as follows:
For $\nu > 0$, $n_0(\nu)$ sufficiently large,
\begin{align}
P(\bar{\theta}(t) \to H| \theta_{n_0} \in B) \geq 1-2de^{-\frac{\hat{K}(\nu)\delta_B^2}{ds(n_0)}}- 2d e^{-\frac{\hat{C}(\nu)\delta_B^2}{ds(n_0)}} - 2\nu.\nonumber
\end{align}The proof will work by selecting a large compact set $C(\nu)$ s.t. $P(\limsup_n \|Y_n\| < C(\nu)) > 1- \nu$ and doing the 
same calculation as in Section \ref{main_res} on this set with probability at least $1-\nu$. 
%Therefore for special kind of Markov iterate-dependent
%process if some moment condition implies $E[\sup_n\|Y_n\|]<\infty$, 
%the above argument will go through.  

\section{Proof of almost sure convergence}
\label{a.s.conv}
\subsection{Almost sure convergence under asymptotic tightness}
\begin{defn}
A sequence of random variables $\{\theta_n\}$ is called asymptotically tight if for each $\epsilon>0$
there exists a compact set $K_{\epsilon}$ such that
\begin{align}
\label{tight}
\limsup_{n \to \infty} P(\theta_n \in K_{\epsilon}) \geq 1-\epsilon.
\end{align}
\end{defn}
Clearly, (\ref{tight}) is a much weaker condition than (\ref{stability}).
In the following, we give a sufficient condition to guarantee the above:
\begin{lem}
If there is a $\phi\geq 0$ so that
$\phi(\theta) \to \infty$ as $\|\theta\| \to \infty$
and
%there exists a $C_0 >0$ such that
\begin{align}
\label{tight_condn}
\liminf_{n \to \infty} E[\phi(\theta_n)] < \infty,
\end{align}
then $\{\theta_n\}$ is asymptotically tight.
\end{lem}
\begin{proof}
Proof is by contradiction and is similar to the proof of sufficient condition for full tightness as given in Theorem 3.2.8 of \cite[p.~104]{durrett}.
\end{proof}
Next, we show that if the stochastic approximation iterates are asymptotically
tight then we can prove almost sure convergence to $H$ under some reasonable assumptions.
\begin{thm}
\label{a.s.}
Under \textbf{(A1)-(A7)}, if $\{\theta_n\}$ is asymptotically tight and $\liminf_n P(\theta_n \in G) =1$ then $P(\theta_n \to H) =1$.
\end{thm}
\begin{proof}
Choose an open $B$ with compact closure such that $H, K_{\epsilon} \cap G \subset B \subset \bar{B}\subset G$.
Therefore
\small
\begin{align}
\nonumber
&\limsup_{n_0 \to \infty}P(\theta_{n_0} \in B)\nonumber \\
                  &\geq \limsup_{n_0\to \infty}P(\theta_{n_0} \in G \cap K_{\epsilon})\nonumber \\
                  &= \limsup_{n_0\to \infty}\left[P(\theta_{n_0} \in K_{\epsilon}) + P(\theta_{n_0} \in G) - P(\theta_{n_0} \in G\cup K_{\epsilon})\right]\nonumber \\
                  &\geq \limsup_{n_0 \to \infty} P(\theta_{n_0} \in K_{\epsilon}) + \liminf_{n_0 \to \infty}P(\theta_{n_0} \in G) - \nonumber \\
                  &\limsup_{n_0 \to \infty} P(\theta_{n_0} \in G\cup K_{\epsilon})\nonumber \\
                  &\geq 1 - \epsilon +1 - 1\nonumber.
\end{align}
\normalsize
Thus there exists a subsequence $n_0(k)$ s.t. $P(\theta_{n_0(k)} \in B)>0$. Now, 
\small
\begin{align}
\nonumber
&\lim_{k \to \infty} P(\theta_{n_0(k)} \in B, \theta_n \to H) \nonumber \\
&= \lim_{k \to \infty} P(\theta_{n_0(k)} \in B) P(\theta_n \to H | \theta_{n_0(k)} \in B) \nonumber \\
&= \lim_{k \to \infty} P(\theta_{n_0(k)} \in B) \mbox{~~using Theorem \ref{main_thm}}\nonumber
\end{align}
\normalsize
Therefore, 
\begin{align}
\nonumber
P(\theta_n \to H) &\geq \limsup_{n_0 \to \infty} P(\theta_n \to H, \theta_{n_0} \in B) \nonumber \\
                  &= \limsup_{n_0 \to \infty} P(\theta_{n_0} \in B) \nonumber \\
                  &\geq 1-\epsilon  
\end{align}
\normalsize 
Now let $\epsilon \to 0$.
\end{proof}
%\begin{cor}
%$\theta_n \to H$ a.s
%\end{cor}
%\begin{proof}
%The proof is based on boosting technique. Repeat the same experiment with independent seed then as a corollary to Second
%Borel-Cantelli lemma we get almsot sure convergence.
%\end{proof}
\begin{rmk}
We compare Theorem \ref{a.s.} to the main convergence result (Kushner-Clark Lemma) from \cite[Section II~C]{benveniste}
where stability of the iterates was assumed. In that case much weaker condition, namely $\theta_n \in A$ infinitely often
where $A$ is some compact subset of $G$ was sufficient to draw
the conclusion. Here we need a much stronger condition such as $\liminf_{n_0}P(\theta_{n_0} \in G)=1$.
\end{rmk}
\begin{rmk}
Theorem  \ref{a.s.} is valid for any `local' attractor $H$ whereas in \cite[Theorem 1]{sameer} $H$ was a `global' attractor.
\end{rmk}

There are sufficient conditions to guarantee tightness (\cite[Chap 6, Theorem 7.4]{kushner}) of the iterates in the literature. 
In the following we describe another set of sufficient conditions which guarantee (\ref{tight_condn}):
\begin{lem}
Suppose there exists a $\phi \geq 0$ with $\phi(\theta) \to \infty$ as $\|\theta\| \to \infty$ and the following properties:
Outside the unit ball
\begin{enumerate}[label=\textbf{(S\arabic*)}]
 \item $\phi$ is twice differentiable and all second order derivatives are bounded
by some constant $c$.
 \item for every $\theta$, $K \subset \mathbb{R}^k$ compact, $\langle \nabla \phi(\theta), f(\theta,y)\rangle \leq 0$ for all $y\in K$.
\end{enumerate}
 Then for the step size sequences of the form $a(n)=\frac{1}{n(\log n)^p}$ with $0< p \leq 1$, $n \geq 2$, we have (\ref{tight_condn}).
\end{lem}

%\begin{rmk}
%This tightness condition will be much more useful when $\{Y_n\}$ are i.i.d. iterate-dependent 
%which can be absorbed together with the martingale noise.
%In that case \textbf{(S2)} can be replaced 
%by 
%\begin{enumerate}[label=\textbf{(S\arabic*)'}]
%\setcounter{enumi}{1}
%\item for every $\theta$, $\langle \nabla \phi(\theta), h(\theta)\rangle \leq 0$.
%\end{enumerate} 
%\end{rmk}

%A Liapunov condition such as $\langle\nabla V, h\rangle < - g$ with $\lim_{\|x\|\uparrow\infty}g(x) = \infty$ 
%is stronger than the one with $-\epsilon$ on RHS, one with $-\alpha V + C$ with $\alpha, C > 0$ constants is even stronger. 
%If you use the one with $g$, can manipulate the growth rate of $G$ by replacing $V$ by $f(V)$ for smooth increasing $f$. 
%This may or may not help. If the dynamics comes from an algorithm where usually we expect a solution in some bounded region, 
%there is freedom to tweak the behavior of the algorithm near infinity, e.g., by adding a strong inward drift for large $\|x\|$, 
%so such conditions are not unreasonable.

\begin{proof}
Following similar steps as in \cite[Theorem 3]{sameer} and \textbf{(S2)} we get
\begin{align}
\label{sameer}
E[\phi(\theta_{n+1})|\mathcal{F}_n] \leq \phi(\theta_n) +ca(n)^2(1+\|\theta_n\|^2)\mbox{~a.s.}.
\end{align}
Now we know that, for $n \geq 1$
\begin{align}
 \begin{aligned}
  \|\theta_n\| &\leq \|\theta_{0}\| + \sum_{k=0}^{n-1} a(k)\left[\|f(\theta_k, Y_{k})\|+\|M_{k+1}\|\right]\nonumber \\
               &\leq \|\theta_{0}\| + \tilde{K} \sum_{k=0}^{n-1} a(k) + \tilde{K} \sum_{k=0}^{n-1} a(k)\|\theta_k\|.
 \end{aligned}
\end{align}
Therefore, using a general version of discrete Gronwall inequality (See Appendix)
and the fact that
$\|\theta_{0}\| + \tilde{K} \sum_{k=0}^{n-1} a(k)$ is an increasing function of $n$, we get that
\begin{align}
\|\theta_n\| \leq  \left[ \|\theta_{0}\| + \tilde{K} \sum_{k=0}^{n-1} a(k)\right]\exp(\tilde{K} \sum_{k=0}^{n-1} a(k)). \nonumber
\end{align}
Thus,
\begin{align}
\label{series}
\nonumber
\liminf_n &E[\phi(\theta_n)] < \phi(\theta_0) +ca(0)^2(1+\|\theta_0\|^2) + c\sum_{n=1}^{\infty} a(n)^2 + \\
&c\sum_{n=1}^{\infty} a(n)^2 [\|\theta_{0}\| + \tilde{K} \sum_{k=0}^{n-1} a(k)]^2\exp(2\tilde{K} \sum_{k=0}^{n-1} a(k)).
\end{align}
In the following, we show that for the mentioned step-size sequence, the R.H.S converges.
Assume $0< p < 1$.
Then

\begin{align}
\nonumber
\sum_{i=2}^{n-1}a(i) &\leq \int_{1}^{n-1} \frac{1}{i(\log i)^p}di \nonumber \\
                     &\leq\frac{1}{1-p}(\log n)^{1-p}.\nonumber
\end{align}

Then,
\begin{align}
\nonumber
\sum_{n=2}^{\infty} \frac{(\log n)^{2(1-p)}}{n^2(\log n)^{2p}} \exp\left[\frac{2\tilde{K}}{1-p}(\log n)^{1-p}\right]\nonumber \\
= \sum_{n=2}^{\infty} \frac{(\log n)^{2-4p}}{n^{2+\frac{2\tilde{K}}{(p-1)(\log n)^p}}}.\nonumber
\end{align}
This is a convergent series for $0<p<1$ as there exists an $\epsilon >0$ such that for large $n$
\begin{align}
(\log n)^{2-4p} \leq n^{1- \frac{2\tilde{K}}{(1-p)(\log n)^p}-\epsilon}.\nonumber
\end{align}
Also, the following series converges:
\begin{align}
\nonumber
\sum_{n=2}^{\infty} \frac{\exp\left[\frac{2\tilde{K}}{1-p}(\log n)^{1-p}\right]}{n^2(\log n)^{2p}}.\nonumber \\
\end{align}
Moreover, it is easy to check that the above arguments also hold for $p=1$.
\end{proof}
Thus we show that \textbf{(A5)} in Theorem 3 of \cite{sameer} is not required for the step size sequence of the form
$a(n)=\frac{1}{n(\log n)^p}$ with $0<p\leq 1$ which is clearly a divergent series but $\sum_{n} a(n)^2 < \infty$.

%An example of such a $g$ can be $g(\theta) = \|\theta - \theta^*\|^{1+\epsilon}, 0<\epsilon\leq1$. It is easy to check that
%\textbf{(A8)} will be satisfied for such $g$ outside the unit ball.
%\begin{rmk}
%Our proof also goes through if \textbf{(S2)} is replaced by 
%\begin{enumerate}[label=\textbf{(S\arabic*)''}]
%\setcounter{enumi}{1}
%\item $\langle \nabla g(\theta_n),f(\theta_n,Y_n)\rangle \leq \frac{k(1+\|\theta_n\|^2)}{n(logn)^p}$ 
%\end{enumerate}
%with $0< p \leq 1$
%\end{rmk}
\begin{rmk}
Note that in \textbf{(S2)} we strengthen the usual Lyapunov condition namely $\langle \nabla \phi(\theta), \int f(\theta,y)\tau_\theta(dy)\rangle \leq 0$. However, we show that for linear stochastic approximation \textbf{(S2)} always gets satisfied in the case of finite state Markov chains. Let us assume that $f(\theta,y) = -A\theta + g(y)$ where $g: S \to \mathbb{R}^d$. and $A$ is a $d \times d$ positive definite matrix. Let $\theta^* = \int g(y) \Gamma(dy)$ where the Markov noise does not depend on the iterate. Let $\phi(\theta) = \frac{1}{2}(\theta - \theta^*)^T(\theta - \theta^*)$. Now, the required condition so that \textbf{(S2)} gets satisfied is that
outside some compact set, for all $\theta \in \mathbb{R}^d, y \in S$
\begin{align}
\langle \theta^*, \theta \rangle + \langle \theta, g(y) \rangle \leq \|\theta\|^2 + \langle \theta^*, g(y) \rangle \nonumber 
\end{align}
For a $n$-state Markov chain with stationary distribution $\pi(i),i=1 \dots n$ the above condition gets 
satisfied if 
\begin{align}
\theta \notin \cup_i \{\theta: &\|A\theta - \frac{(1+ \pi(i))g(i) + \sum_{j=1, j \neq i}^n \pi(j)g(j)}{2}\| \leq \nonumber \\ &\frac{\|(1-\pi(i))g(i) - \sum_{j=1, j \neq i}^n \pi(j)g(j)\|}{2} \}\nonumber 
%\{\theta: \|\theta - \frac{(1+\alpha) g(a) + (1-\alpha) g(b)}{2}\| \leq \nonumber \\ 
%&(1-\alpha)\|g(a) - g(b)\|\} \cup \nonumber \\ 
%&\{\theta: \|\theta - \frac{(1+\beta) g(b) + (1-\beta) g(a)}{2}\| \leq \nonumber \\ 
%& (1-\beta)\|g(a) - g(b)\|\} \nonumber
\end{align}
for which a sufficient condition is
\begin{align}
\theta \notin \cup_i \{\theta: \|\theta\| \leq \frac{\|u\| + \|v\|}{\lambda_1} \}\nonumber 
%\{\theta: \|\theta - \frac{(1+\alpha) g(a) + (1-\alpha) g(b)}{2}\| \leq \nonumber \\ 
%&(1-\alpha)\|g(a) - g(b)\|\} \cup \nonumber \\ 
%&\{\theta: \|\theta - \frac{(1+\beta) g(b) + (1-\beta) g(a)}{2}\| \leq \nonumber \\ 
%& (1-\beta)\|g(a) - g(b)\|\} \nonumber
\end{align}
where $u=\frac{(1+ \pi(i))g(i) + \sum_{j=1, j \neq i}^n \pi(j)g(j)}{2}, v=\frac{\|(1-\pi(i))g(i) - \sum_{j=1, j \neq i}^n \pi(j)g(j)\|}{2},\lambda_1>0$ is the smallest eigenvalue of $A$.
%assuming that $\sum_{i=1}^n \pi(i)=1$
\end{rmk}
\begin{rmk}
Theorem 3 of \cite{sameer} imposes assumptions on the strict Lyapunov function $V(.)$
for the attractor $H$ to ensure tightness of the iterates.
For that reason $H$ is required to be a global attractor there. 
However, we observe that $\phi(.)$ can be different from $V(.)$ because
we only require properties like \textbf{(S2)} to ensure tightness of the iterates.
\end{rmk}
\begin{rmk}
Note that the series in the R.H.S of (\ref{series}) won't converge if $a(n) = \frac{1}{n^k}$ with $1/2 < k \leq 1$.
In such a case \textbf{(A5)} from  \cite{sameer} will be required.
\end{rmk}
%\begin{rmk}
%Note that (\ref{stability}) may not be a stronger condition than (\ref{tight}) in general (as the bound in (\ref{stability}) may
%depend on $\omega$), but we show that the sufficient condition
%for the latter can be found easily compared to the former.
%\end{rmk}

%\subsection{Almost sure convergence for stochastic fixed point iteration}
\subsection{Proof of stability and a.s. convergence using our results}
\begin{comment}
If the iteration is a stochastic fixed point one with the function being a
contraction, it is not clear whether the assumptions
to ensure stability of the iterates \cite[Chap. 6.3]{borkar_book} are satisfied.
Therefore for this special case
even the assumptions ensuring asymptotic tightness are not required.
\end{comment}
%In case where the sufficient conditions for Theorem \ref{a.s.} are not satisfied,
Note that if the iterates belong to some arbitrary compact set (depending on the sample point) infinitely often,
it may not imply stability if the time interval between successively visiting it runs to infinity.
 We show that this does not happen if the compact set and the step-size have special properties.
Using the lock-in probability results from Section \ref{main_res}, we prove stability and therefore convergence
of the iterates
on the set $\{\theta_n \in B \mbox{~i.o.}\}$ when the step-size is $a(n) = \frac{1}{n^k}, \frac{1}{2} <k \leq 1$.

Consider the settings described in Section \ref{main_res}.
Let $A = \{\omega: \exists m \geq 0 \mbox{~~s.t.~~} \rho_m(\omega) \geq \delta\}$.
Then Theorem \ref{main_thm} shows that for sufficiently large $n_0$,
\begin{align}
P(A | \theta_{n_0}\in B) < 4de^{-\frac{C}{s(n_0)}} \nonumber \\
%\implies P(A_{n_0} | \theta_{n_0}\in B) I_{\theta_{n_0}\in B)} < 4de^{-\frac{C}{s(n_0)}} \nonumber \\
\implies P(A \cap \{\theta_{n_0}\in B\}) < 4de^{-\frac{C}{s(n_0)}} \nonumber \\
\implies \sum_{n_0 =1}^{\infty} P(A \cap \{\theta_{n_0}\in B\}) < \sum_{n_0 =1}^{\infty} 4de^{-\frac{C}{s(n_0)}}.\label{rhs}
\end{align}
Now, for $n \geq 2$,
\begin{align}
s(n) = \sum_{i=n}^{\infty} \frac{1}{i^{2k}} &< \int_{i=n-1}^{\infty} \frac{1}{i^{2k}} di \nonumber \\
                                         &= \frac{1}{(2k-1)(n-1)^{2k-1}} \nonumber \\
                                         &\leq \frac{1}{(2k-1)(\frac{n}{2})^{2k-1}}\nonumber
\end{align}
Now, for large $n$,
\begin{align}
e^{(2k-1) (\frac{n}{2})^{2k-1}} > n^2.\nonumber
\end{align}

Therefore, the R.H.S in (\ref{rhs}) is finite for the mentioned step-size. The same argument follows for the
step-size $\frac{1}{n (\log n)^{k}}, k\leq 1$ as for large $n$, $(\log n) ^{2k} \geq 1$.

Therefore,
\begin{align}
E[\sum_{n_0 =1}^{\infty} I_{A \cap \{\theta_{n_0}\in B\}}] < \infty \nonumber \\
\implies  I_{A} \sum_{n_0 =1}^{\infty}  I_{\{\theta_{n_0}\in B\}} < \infty \mbox{~a.s.}\nonumber
\end{align}
Therefore on the event $\{\theta_{n_0}\in B \mbox{~i.o}\}$, $I_A =0 \mbox{~a.s.}$ which is nothing but
$\sup_n \|\theta_n\| < \infty$ a.s. The result can be summarized as follows:
\begin{cor}
Under the assumptions made in Section \ref{sec_def} and the following additional assumptions:
\begin{enumerate}[label=\textbf{(W\arabic*)}]
 \item $\forall N,~~~\exists n \geq N$ s.t. $P(\theta_n \in B) > 0$ where $B$ is chosen as in Section \ref{sec_def}, 
 \item $\sum_{n=1}^{\infty} P(\theta_n \in B | \mathcal{F}_{n-1}) = \infty \mbox{~~a.s.}$,
\end{enumerate}
we have 
\begin{align}
\sup_n \|\theta_n\| < \infty \mbox{~~ a.s. and ~~}  \theta_n \to H \mbox{~~ a.s.}, \nonumber
\end{align}
for the step-size sequence of the form $a(n)=\frac{1}{n^k}, 0.5 < k \leq 1$ and $\frac{1}{n(logn)^k}, k \leq 1$. 
\end{cor}

\begin{comment}
\begin{rmk}
We compare our result to the statement of Kushner-Clarke lemma \cite[Section III E]{benveniste} which assumes stability of the
iterates and $h$ to be continuous.
%Therefore the interpolated algorithm trajectory is tracked by the o.d.e asymptotically.
Thus
the  only requirement to prove
almost sure convergence is that  the iterates should
enter some compact set in the domain of attraction of $H$ infinitely often.
%Although
%According to Kushner-Clarke lemma, on the set
%where iterates are stable, if the iterates belong to some compact set $A \in G$ i.o. then  $\theta_n \to H$ almost surely  whereas
%we need the iterates to enter some open set with compact closure  infinitely often which \textit{contains} $H$ and is
%in the domain of attraction of $H$
Although we need that $h$ is Lipschitz continuous and
the iterates belong to some open set infinitely often
such that the open set \text{contains} the local attractor and its compact closure is in the
domain of attraction of the local attractor, we do not assume stability of the iterates to prove a.s. convergence.

%A similar result was proved in \cite[Theorem 2.1]{lock_in_original} for
%Lipschitz vector field and local attractor although it used stability of the iterates.
\end{rmk}
\end{comment}
%\begin{rmk}
%This can be useful in off-policy reinforcement learning scenarios with function approximation (e.g. Q learning) where
%proving stability of the iterates  and therefore may not converge.
%\end{rmk}
%\begin{comment}
\subsection{Comparison of our results with related literature}
\label{comp}
In this section we compare our results and assumptions with the related literature. In \cite{andrieu1}, the 
truncations
on adaptive truncation sets from \cite{fu_chen} has been extended to the case where the noise is Markov. 
It is mentioned there that the procedure they follow is 
different
in some respects from the original
procedure proposed by \cite{fu_chen}. To prove that the number of re-initializations of the procedure described in 
\cite[Section 3.2]{andrieu1} is finite, they establish a bound on the probability that the $n$-th reinitialization time is 
finite in terms of the fluctuations of the noise sequence of the algorithm between successive re-initializations.
Here we compare our assumptions with the assumptions made in \cite{andrieu1} in detail. \textbf{(A1)} therein assumes existence of a global attractor $L$ (say the corresponding Lyapunov function is $w$) whereas our results are true for local attractors. The convergence proof assuming stability \cite[Theorem 2.3]{andrieu1} therein heavily depends on the assumption on global attractor. Moreover, \cite{andrieu1} additionally assumes that there exists $M_0 > 0$ such that 
$L \subset \{\theta: w(\theta) < M_0\}$. This kind of assumption gets satisfied if the Lyapunov function is quadratic. It is not clear whether this assumption gets satisfied for non-quadratic Lyapunov functions. Now consider the assumption that there exists $M_1 \in (M_0 , \infty]$ such that $\{\theta: w(\theta) \leq M_1 \}$ is a compact set. Clearly, this is a closed set, however, there is no gurantee that this will be bounded. Next, look at the assumption that the closure of $w(L)$ has an empty interior. Using Sard’s theorem one can say that if $w$ is
$d$-times continuously differentiable, then $w(\{\nabla w = 0\})$ has an empty interior. It is not clear whether this condition also gets satisfied in the case of non-quadratic Lyapunov function. In order to control the fluctuations some less classical assumptions have been imposed on the transition kernel (regularity properties in $V$ and $V^p$ norm) as well 
as on the vector field (see (DRI2) and (DRI3) and the discussion thereafter) whereas our assumptions on Poisson equation as well as transition kernel are classical.

In case when the noise is markov if one tries to naturally extend the classical Borkar-Meyn theorem \cite{borkar_meyn}, the following problem arises in \cite[Chapter 6, Lemma 5]{borkar_book}. Let $\theta^n(\cdot), n=1,2,\dots$ denote solutions to the following o.d.e 
\begin{align}
\dot{\theta}(t)=\tilde{h}_c(\theta(t), \mu(t)) 
\end{align}
with $c=r(n)$ and $\mu(.)$ replaced by $\mu^n(.)$ where $h_c(\theta,y) = \frac{h(c\theta,y)}{c}$. $\theta^{\infty}(\cdot)$ denote the solutions of the same o.d.e with $\mu(.)$ replaced by $\mu^{\infty}(\cdot)$ and $c=\infty$ where  $h_{\infty}(\theta,y) = \lim_{c \to \infty} \frac{h(c\theta,y)}{c}$. Now it is easy to see that to prove that the scaled trajectory is tracked by the infinity system one needs to consider the following difference
\begin{align}
\|\theta^n(t) - &\theta^{\infty}(t)\| \leq \|\theta^n(0) - \theta^{\infty}(0)\| + \nonumber\\
                                     &\int_{0}^{t}\| \tilde{h}_{r(n)}(\theta^n(s), \mu^n(s))- \tilde{h}_{r(n)}(\theta^{\infty}(s), \mu^n(s))\|ds + \nonumber\\ 
                                     &\|\int_{0}^{t} \tilde{h}_{r(n)}(\theta^{\infty}(s), \mu^n(s))- \tilde{h}_{\infty}(\theta^{\infty}(s), \mu^n(s))ds\| + \nonumber\\ 
                                     &\|\int_{0}^{t} \tilde{h}_{\infty}(\theta^{\infty}(s), \mu^n(s))- \tilde{h}_{\infty}(\theta^{\infty}(s), \mu^{\infty}(s))ds\|
\end{align}
Therefore one needs to assume that for every $T$,
\begin{align}
\sup_{t \in [0,T]}  \|\int_{0}^{t} \tilde{h}_{r(n)}(\theta^{\infty}(s), \mu^n(s))- \tilde{h}_{\infty}(\theta^{\infty}(s), \mu^n(s))ds\| \to 0 \nonumber
\end{align}. Further one needs to assume 1) Lipschitz continuity of $h$ in the first component \textit{uniformly} w.r.t 
the second and 2) $h$ being jointly continuous.

One recent work \cite{arun_markov} finds sufficient conditions for stability  (almost  sure  boundedness) of stochastic approximation algorithms (SAs) driven by a `controlled Markov' process under a new set of assumptions compared to the stability criteria for the martingale noise case \cite{borkar_meyn} whereas our tightness conditions for Markov noise are clearly stronger than the corresponding martingale noise case \cite{sameer}, therefore the extension is more justifiable than the extension considered in \cite{arun_markov}. Moreover, our assumptions on vector field (see \textbf{(A2)}) are weaker than the assumptions in \cite{arun_markov} (See \textbf{(A1)}). For example, consider the linear stochastic approximation case presented in Section \ref{a.s.conv} as an example to satisfy the tightness condition. If $g(.)$ is a bounded discontinuous function then \textbf{(A2)} gets satisfied however \textbf{(A1)}) from \cite{arun_markov} does not.

Another related work is \cite{yaji_lock} where the behavior of stochastic approximation schemes with set-valued maps
in the absence of a stability guarantee is analyzed, however this work does not consider Markov noise in its analysis. 

\section{Lock-in probability calculation for iterates with different timescales: tracking ability of ``general'' adaptive algorithms using lock-in probability}
\label{track}
In this section we investigate the tracking ability of algorithms of the type:
\begin{align}
\label{fast_alg}
w_{n+1} &= w_n + b(n)\left[g(\theta_n, w_n, Z^{(2)}_n) + M^{(2)}_{n+1}\right],
\end{align}
that are driven by a ``slowly'' varying single timescale stochastic approximation process:
\begin{align}
\label{slow_process}
\theta_{n+1} = \theta_n + a(n)\left[h(\theta_n,Z^{(1)}_n) + M^{(1)}_{n+1}\right], 
\end{align}
when none of the iterates are known to be stable. Here, 
$\theta_n \in \mathbb{R}^d, w_n \in \mathbb{R}^k, Z^{(1)}_n \in \mathbb{R}^l, Z^{(2)}_n \in \mathbb{R}^m$.
Note that there is a unilateral coupling between (\ref{fast_alg}) and (\ref{slow_process}) in that 
(\ref{fast_alg}) depends on (\ref{slow_process}) but not the other way.
Suppose $w_n$ converges to a function $\lambda(\theta)$ in case $\theta_n$ is kept constant at $\theta$, then an 
interesting  question is that if $\theta_n$ changes slowly, 
can $w_n$ track the changes in $\theta_n$, i.e., what can we say about the quantity $\|w_n-\lambda(\theta_n)\|$ in the limit.
As mentioned in \cite{konda_linear} such algorithms may arise in the context of adaptive algorithms.
However, in that work, tracking was proved under the restrictive assumption that the stochastic approximation driven by the slowly varying process 
is linear (see (1) in the same paper) and the underlying Markov process in the faster iterate 
is driven by only the slow iterate. 
Using the lock-in probability results of Section \ref{main_res} 
we prove convergence as well as tracking ability of much general algorithms such as 
(\ref{fast_alg})-(\ref{slow_process}) under the following assumptions (we also 
give a detailed comparison with  the assumptions of \cite{konda_linear}):
\begin{enumerate}[label=\textbf{(B\arabic*)}]
 \item $h, Z^{(1)}_n$ and $M^{(1)}_{n+1}$ satisfy the same assumptions satisfied by  
similar quantities ($f, Y_n, M_n$ respectively) of Section \ref{sec_def}.
$g$ satisfies the following assumption:
$\sup_{z}\|g(\theta,w,z)\| \leq K_1 (1+\|\theta\|+ \|w\|+ \|z\|)$ for all $\theta,w,z$ where $K_1 >0$.
Additionally, $\hat{g}(\theta,w) = \int g(\theta, w,z)\Gamma^{(2)}_{\theta,w}(dz)$ is Lipschitz continuous, 
$\Gamma^{(2)}_{\theta,w}$ being the unique stationary distribution of $Z^{(2)}_n$ for a fixed $(\theta,w)$ pair. 
\begin{rmk} In 
(1) of \cite{konda_linear}, the vector field in the faster iterate is linear in the faster iterate variable. Also, the slower
iterate is not a stochastic approximation iteration there.
\end{rmk}
 \item $\{a(n)\}$ is as in \textbf{(A3)}. $\{b(n)\}$ satisfies similar assumptions as $\{a(n)\}$. 
Additionally, $a(n) < b(n) < 1$ for all $n$ and $\frac{a(n)}{b(n)} \to 0$. 
\begin{rmk}
The latter is a much 
weaker requirement than Assumption 4 of \cite{konda_linear}. 
\end{rmk}
 \item The dynamics of $Z^{(2)}_n, n\geq 0,$ is specified by
\begin{align}
P(Z^{(2)}_{n+1} \in B |Z^{(2)}_m, \theta_m, w_m, m\leq n) = \int_{B} \Pi^{(2)}_{\theta_n, w_n}(Z^{(2)}_n; dz), \nonumber \\ 
\mbox{~a.s.,} n\geq 0, \nonumber 
\end{align} 
for $B$ Borel in $\mathbb{R}^m$. 
Assumptions similar to \textbf{(A1)}, \textbf{(A4)}, \textbf{(A6)} and \textbf{(A7)} will be true in the case of $Z^{(2)}_{n}$ also 
with the exception that now $\theta$ will be replaced by the tuple $(\theta,w)$. 
\begin{rmk}
In \cite{konda_linear}, 
the Markov process depends on only the slow parameter.
\end{rmk}
 \item $\{M^{(i)}_n\}, i=1, 2$ are martingale difference sequences
w.r.t the increasing $\sigma$-fields
\begin{align}
\mathcal{F}_n = \sigma(\theta_m, w_m, M^{(i)}_{m}, Z^{(i)}_m, m \leq n, i=1,2), n \geq 0,\nonumber 
\end{align}
where $M^{(2)}_n, n\geq 0,$ satisfy the following:
\begin{align}
\|M^{(2)}_{n+1}\| \leq K_2(1 + \|\theta_n\| + \|w_n\|), K_2>0.\nonumber 
\end{align}. 
\begin{rmk}
Our assumptions on martingale difference noise 
are stronger than the same in \cite{konda_linear} (See Assumption 5). 
\end{rmk}
\item The o.d.e
\begin{align}
\dot{w}(t) = \hat{g}(\theta,w(t)) \nonumber
\end{align}
has a global attractor $\lambda(\theta)$ with $\lambda: \mathbb{R}^d \to \mathbb{R}^k$ being Lipschitz 
continuous.

The o.d.e 
\begin{align}
\label{slow}
\dot{\theta}(t)=\hat{h}(\theta(t)) 
\end{align}
has an asymptotically stable set $H^s$ with domain of attraction $G^s$ 
where $\hat{h}(\theta) = \int h(\theta,y) \Gamma^{(1)}_\theta(dy)$ is Lipschitz continuous with 
$\Gamma^{(1)}_\theta$ is the same as $\Gamma_\theta$ in \textbf{(A4)}.

For every compact set $C_1 \subset \mathbb{R}^d$, the set 
$\{(\theta,\lambda(\theta)): \theta \in C_1\}$ is Lyapunov stable.
\item The iterates $\{\theta_n,w_n\}$ are asymptotically tight (for which a sufficient condition is stated later). 
\begin{rmk}
In 
\cite{konda_linear} an important step in the proof is the proof of the stability of the iterates.
\end{rmk}
\end{enumerate}
%Let $V_1$ be the Lyapunov function for the local attractor of the o.d.e (\ref{slow}) and $V_2$ be the 
%Lyapunov function for the local attractor $\{(\theta,\lambda(\theta)): \theta \in C_1\}$ of the coupled o.d.e
%\begin{align}
%\label{couple}
%\dot{\alpha}(t) = \hat{G}(\alpha(t)) 
%\end{align}
\begin{rmk}
A recent work \cite{borkar_two} provides a sample complexity estimate for two time-scale stochastic approximation using Alekseev formula, however, under the assumption that the vector fields as well as $\lambda$ being continuously differentiable and the attractors are single points. 
\end{rmk}
\begin{thm}
Under the above assumptions, for sufficiently large $n_0$,
\begin{align}
&P((\theta_n,w_n) \to \bigcup_{\theta \in H^s} (\theta, \lambda(\theta))|\theta_{n_0} \in B_1, w_{n_0} \in B_2). \nonumber \\ 
%&\geq \left(1-2de^{-\frac{\hat{K^s}{\delta_{B_1}}^2}{dS_1(n_0)}}- 2d e^{-\frac{\hat{C^s}{\delta_{B_1}}^2}{dS_1(n_0)}}\right)\left(1 -\frac{2ke^{-\frac{\hat{K^c}{\delta_{B_2}}^2}{kS_2(n_0)}}+2k e^{-\frac{\hat{C^c}{\delta_{B_2}}^2}{kS_2(n_0)}}}{1-2de^{-\frac{\hat{K^s}{\delta_{B_1}}^2}{dS_1(n_0)}}- 2d e^{-\frac{\hat{C^s}{\delta_{B_1}}^2}{dS_1(n_0)}}}\right)\nonumber 
&\geq \left(1-o(S_1(n_0))\right)\left(1-\frac{o(S_2(n_0))}{1-o(S_1(n_0))-o(S_2(n_0))}\right)\nonumber
\end{align} 
%$\hat{K^c}, \hat{C^c},\hat{K^s}, \hat{C^s}$ are versions of the constants $\hat{K}, \hat{C}$ of Section \ref{main_res} for 
%fast and slow timescale respectively. 
\end{thm}
\begin{proof}
Let there be an open set $B_1$ with compact closure such that $H^s \subset B_1 \subset \bar{B_1} \subset G^s$.
From the results of Section \ref{main_res}, we can find a $T^s$ such 
that any trajectory for the o.d.e (\ref{slow}) starting in $\bar{B_1}$ will be within some $\epsilon_1$-neighborhood of $H^s$ after time $T^s$. Let,  $S_1 = \left[\sup_{\theta \in \bar{B_1}}\|\theta\| +\tilde{K}\right] e^{\tilde{K}T^s}$, 
%\hat{G}(\alpha) = \int G(\alpha,z)\Gamma^2(\alpha)(dz)$ and $G(\alpha,z)= (0,g(\alpha,z))$. 
and $C_1=\{\theta: \|\theta\| \leq S_1\}$. 
Let there be an open set $B_2$ with compact closure such that 
$\lambda(C_1) \subset B_2 \subset \bar{B_2} \subset \mathbb{R}^k$. Choose $\delta_{B_1}$ 
in the same way that $\delta_B$ is chosen in Section \ref{main_res}. Choose $\delta_{B_2}, 0<\epsilon''_1< \epsilon''$ such that 
$N_{\delta_{B_2} + \epsilon''_1}(\lambda(C_1)) \subset  N_{\epsilon''}(\lambda(C_1)) \subset B_2$. 
If the coupled o.d.e starts at a point such that its $\theta$ and $w$ co-ordinates are in $C_1$ and $\bar{B_2}$ 
respectively then as in Section \ref{main_res} one can find a $T^f >0$ (independent of the starting point) 
that is the maximum time required for the 
the o.d.e to be in the $\epsilon''_1$-neighbourhood of $\{(\theta,\lambda(\theta)): \theta \in C_1\}$. 
%Also 
%results of Section \ref{main_res} shows that if the o.d.e (\ref{slow}) starts within $\bar{B_1}$ then 
%after time $T^s_0$ it will always be within $C_1$. 
Now, 
let $T^c = \max(T^f, T^s+1)$ and 
for $m \geq 1$ define,   
%Then as in Section \ref{main_res} 
%one can find threshold times $T^s$ and $T^c$ (which was $T$ in Section \ref{main_res}) for slow and 
%fast timescales from the respective Lyapunov functions. Let, 
\begin{align}
n^c_0 = n^s_0 = n_0, \nonumber \\
t^c(n) = \sum_{i=0}^{n-1} b(i),  n^c_m = \min\{n: t^c(n) \geq t^c(n^c_{m-1}) + T^c\}, \nonumber \\
t^s(n) = \sum_{i=0}^{n-1} a(i),  n^s_m = \min\{n: t^s(n) \geq t^s(n^s_{m-1}) + T^s\}. \nonumber
\end{align}
Similarly, for $m \geq 0$, define 
\begin{align}
T^c_m = t^c(n^c_m), I^c_m = [T^c_m, T^c_{m+1}], \\
T^s_m = t^s(n^s_m), I^s_m = [T^s_m, T^s_{m+1}], \\
l_m = \max(k: t^s(n^s_k) \leq t^c(n^c_m)). \nonumber
\end{align}
Now define, 
\begin{align}
\rho^c_m:=\sup_{t\in I^c_m}\|\bar{\alpha}(t) - \alpha^{T^c_m}(t)\| \nonumber
\end{align}
where $\bar{\alpha}(.)$ is the interpolated trajectory for the coupled iterate 
\begin{align}
\label{coupled}
\alpha_{n+1} = \alpha_n + b(n)\left[G(\alpha_n,Z^{(2)}_n) + \epsilon'_n +  M^{(4)}_{n+1}\right]
\end{align}
where $\alpha_n=(\theta_n, w_n)$. Let  $\epsilon_n = \frac{a(n)}{b(n)}h(\theta_n, Z^{(1)}_n)$ and $M^{(3)}_{n+1} = \frac{a(n)}{b(n)} M^{(1)}_{n+1}$ for $n\geq 0$. 
Now let $\alpha=(\theta,w) \in \mathbb{R}^{d+k}, G(\alpha,z)=(0, g(\alpha,z)), \epsilon'_n=(\epsilon_n, 0), 
M^{(4)}_{n+1}= (M^{(3)}_{n+1}, M^{(2)}_{n+1})$,
and $\alpha^{T^c_m}(.)$ is the solution of the o.d.e 
\begin{align}
\dot{w}(t) = \hat{g}(\theta(t),w(t)), \dot{\theta}(t) = 0,\nonumber 
\end{align} on $I^c_m$ with the initial point $\alpha^{T^c_m}(T^c_m) = \bar{\alpha}(T^c_m)$.
Also, define 
\begin{align}
\rho^s_m:=\sup_{t\in I^s_m}\|\bar{\theta}(t) - \theta^{T^s_m}(t)\| \nonumber
\end{align}
where $\theta^{T^s_m}(.)$ denotes the solution of the o.d.e (\ref{slow}) on $I^s_m$ with the initial point $\theta^{T^s_m}(T^s_m)=\bar{\theta}(T^s_m)$. 
Let us assume for the moment that $\theta_{n_0} \in B_1, w_{n_0} \in B_2$, and that $\rho^s_m < \delta_{B_1}$ and 
$\rho^c_m < \delta_{B_2}$.  
 Then using similar arguments as in Section \ref{main_res}, one can show that 
$\sup_{t \geq T^c_0} (\bar{\theta}(t),\bar{w}(t))< \infty \mbox{~~a.s.}$. 
Further, the sequence of types $(\theta_n,w_n), n\geq 0$ infinitely often visits the 
compact set $C_1 \times \bar{B_2}$ which is in the domain of attraction  
$C_1 \times \mathbb{R}^d$ of the set $\{(\theta, \lambda(\theta)) :  \theta \in C_1\}$. Therefore,
%for $t \geq T^c_1$ and $\bar{\theta}(t) \in B_1$ for $t \geq T^s_1$  
\begin{align}
(\theta_n, w_n) \to \{(\theta, \lambda(\theta)) :  \theta \in \mathbb{R}^d\} \mbox{~~a.s.~~} \nonumber 
\end{align}This, in turn, implies that $\|w_n - \lambda(\theta_n)\| \to 0$ a.s. which implies that 
$(\theta_n, w_n) \to \bigcup_{\theta \in H^s} (\theta, \lambda(\theta))$. 
Let $\mathcal{B}^s_{m}$ denote the event that $\theta_{n_0} \in B_1, w_{n_0} \in B_2$ and $\rho^s_k < \delta_{B_1}$ for $k=0,1,\dots, 
m$. Also, let $\mathcal{B}'_{m,k}$ denote the event that  $\theta_{n_0} \in B_1, w_{n_0} \in B_2$, $\rho^c_j < \delta_{B_2}$ for $j=0,1,\dots, 
m$ and $\rho^s_j < \delta_{B_1}$ for $j=0,1,\dots,
k$.
Therefore, 
\begin{align}
&P((\theta_n,w_n) \to \bigcup_{\theta \in H^s} (\theta, \lambda(\theta))|\theta_{n_0} \in B_1, w_{n_0} \in B_2) \nonumber \\
&\geq  P\left[\rho^c_m < \delta_{B_2} \forall m \geq 0, \rho^s_m < \delta_{B_1} \forall m \geq 0  | \theta_{n_0}\in B_1, w_{n_0} \in B_2
 \right] \nonumber \\
%& \geq 1- P(\exists m \geq 0 \mbox{~~s.t.~~} \rho^c_m \geq \delta_{B_2}|\theta_{n_0} \in B_1, w_{n_0} \in B_2) - P(\exists m \geq 0 \mbox{~~s.t~~} \rho^s_m \geq \delta_{B_1}|\theta_{n_0} \in B_1, w_{n_0} \in B_2)\nonumber \\ 
& \geq P\left[\rho^s_m < \delta_{B_1} \forall m \geq 0 | \theta_{n_0}\in B_1, w_{n_0} \in B_2\right] \times \nonumber \\  
& P\left[\rho^c_m < \delta_{B_2} \forall m \geq 0 | \theta_{n_0}\in B_1, w_{n_0} \in B_2, \rho^s_m < \delta_{B_1} \forall m \geq 0\right] \nonumber \\
& \geq \left[1 - \sum_{m=0}^{\infty} P(\rho^s_m > \delta_{B_1}| \mathcal{B}^s_{m-1})\right] \times \nonumber \\  
&P\left[\rho^c_m < \delta_{B_2} \forall m \geq 0 | \theta_{n_0}\in B_1, w_{n_0} \in B_2, \rho^s_m < \delta_{B_1} \forall m \geq 0\right]\label{lockin}.
\end{align}
Now, using the simple fact that $P(A|BC) \leq \frac{P(A|B)}{P(C|B)}$,
\begin{align}
&P\left[\rho^c_m < \delta_{B_2} \forall m \geq 0 | \theta_{n_0}\in B_1, w_{n_0} \in B_2, \rho^s_m < \delta_{B_1} \forall m \geq 0\right] \nonumber \\
&\geq \left[1- \sum_{m=0}^{\infty}\frac{P(\rho^c_m > \delta_{B_2}|\mathcal{B}'_{m-1,l_m-1})}{P\left[\rho^s_k < \delta_{B_1} \forall k \geq l_m | 
\mathcal{B}'_{m-1,l_m-1}\right]}\right] \nonumber \\
&=\left[1- \sum_{m=0}^{\infty}\frac{P(\rho^c_m > \delta_{B_2}|\mathcal{B}'_{m-1,l_m-1})}{1- f(m)- g(m)}\right]\label{fastrho}
%P\left[\rho^s_m < \delta_{B_1} \forall m \geq 0 | \theta_{n_0}\in B_1, w_{n_0} \in B_2, \rho^c_m < \delta_{B_2} \forall m \geq 0 \right] 
\end{align}
where $$f(m) = P(\rho^s_{l_m} > \delta_{B_1}|\mathcal{B}'_{m-1,l_m-1})$$ 
and $$g(m) = \sum_{k=l_m+1}^{\infty} P\left[\rho^s_k > \delta_{B_1} | \mathcal{B}'_{m-1,k-1}\right].$$
Clearly, $\mathcal{B}'_{m-1,l_m-1} \in \mathcal{F}_{n^c_m}$ and $\mathcal{B}'_{m-1,k-1} \in \mathcal{F}_{n^s_k}$ for all 
$k\geq l_m+1$. However, $\mathcal{B}'_{m-1,l_m-1} \notin \mathcal{F}_{l_m}$. 
Therefore, the tedious task is to calculate an upper bound for $f(m)$. We describe the procedure for doing so in detail.
Now, due to the way $T^c$ is chosen 
\begin{align}
&f(m) \leq \frac{P(\rho^s_{l_m} > \delta_{B_1}|\mathcal{B}'_{m-2, l_m-1})}
{1- \frac{P(\rho^c_{m-1} > \delta_{B_2} | \mathcal{B}'_{m-2, l_{m-1}-1})}{
1- f(m-1) -
\sum_{k=l_{m-1}+1}^{l_m-1} h(k)}} \label{recurse} 
\end{align}
where $$h(k) = P(\rho^s_{k} > \delta_{B_1} | \mathcal{B}'_{m-2,k-1}).$$

Let $S_1(n_0) = \sum_{i=n_0}^{\infty} a(i)^2$ and $S_2(n_0) = \sum_{i=n_0}^{\infty} b(i)^2$.

From (\ref{recurse}) we can see that 
\begin{align}
f(m) \leq \frac{o(S_1(n_0))}{1- \frac{o(S_2(n_0))}{1-f(m-1) - o(S_1(n_0))}}. \nonumber  
\end{align}
 
One can recursively calculate the expression. At the bottom level one calculates the following expression:
\begin{align}
1-P(\rho^s_{l_{1}-1} > \delta_{B_1} |\mathcal{B}^s_{l_{1}-2}).\nonumber 
\end{align}

Using the fact that $S_1(n_0) < S_2(n_0)$ we see from the above that for all $m \geq 0$, $f(m) \leq o(S_2(n_0))$. 
One can easily show using the technique of Section \ref{main_res} 
that for all $m \geq 0$, $g(m) \leq o(S_1(n_0))$. 
%where $B^c_m := \{\theta_{n_0} \in B_1, w_{n_0} \in B_2, \rho^c_k < \delta_{B_2},  \forall k =0,1,\dots,m, \rho^s_k < \delta_{B_1} \forall k =0,1,\dots,[n^c_m]\} \in \mathcal{F}_{n_{m+1}}, 
% B^s_m = \{\theta_{n_0} \in B_1, w_{n_0} \in B_2, \rho^s_k < \delta_{B_1} \forall k =0,1,\dots,m\} \in \mathcal{F}_{n_{m+1}}$.

Handling the first term in the last inequality of  (\ref{lockin}) is exactly the same as in Section \ref{main_res}. 
The numerator of the term inside the summation in (\ref{fastrho}) can also be handled 
in a similar manner except the fact that the additional error $\epsilon'_n$ in (\ref{coupled}) can be made negligible on $\mathcal{B}'_{m-1,l_m-1}$
using the stability of the iterates there  over 
$T^c$ length intervals (the latter can be proved as in Lemma \ref{T_stability}). $n_0$ will be the maximum of 
its versions arising to handle these two parts.
\end{proof}

\begin{rmk}
For the case of Section \ref{sec_def}, i.e., $\theta_n = \theta$ for all $n$, either 1) $a(n)=0$ or 2) $h(\theta_n, Z^{(1)}_{n}) + M^{(1)}_{n+1} =0$ for all $n$. 
%The latter will also imply that .
Further, all the assumptions $(\mathbf{B1}) -  (\mathbf{B6})$
are satisfied and we can recover the results of Section \ref{main_res} by observing that either 1) $S_1(n_0) =0$ or 2) $M^{(1)}_n =0$ for all $n$ (follows 
from the fact that $\{M^{(1)}_n\}$ is a martingale difference sequence). 
\end{rmk}

From this one can easily prove almost sure convergence under tightness.
\begin{thm}
	\label{adapt}
	Under \textbf{(B1)-(B6)}, if $\{\alpha_n\}$ is asymptotically tight and $\liminf_n P(\theta_n \in G^s) =1$ then $P((\theta_n,w_n) \to \bigcup_{\theta \in H^s} (\theta, \lambda(\theta))) =1$ ,i.e.,  
	$\|w_n - \lambda(\theta_n)\| \to 0$ a.s.
\end{thm}

The sufficient conditions for tightness can be derived in the exact
same way as in Section \ref{a.s.conv}. 
\begin{lem}
Suppose there exists a $V': \mathbb{R}^{d+k} \to [0,\infty)$ such that $V'(\alpha) \to \infty$ as $\|\alpha\| \to \infty$ and the following properties hold:
Outside the unit ball,
\begin{enumerate}[label=\textbf{(S\arabic*)}]
 \item $V'$ is twice differentiable and all second order derivatives are bounded
by some constant $c$.
 \item For every $\alpha \in \mathbb{R}^{d+k}$, $K \subset \mathbb{R}^l$ compact, $\langle {(\nabla V'(\alpha))}_{1 \dots d}, h({(\alpha)}_{1 \dots d},z)\rangle \leq 0$ for all $z\in K$.
 \item for every $\alpha \in \mathbb{R}^{d+k}$, $K \subset \mathbb{R}^m$ compact, $\langle {(\nabla V'(\alpha))}_{d+1 \dots d+k}, g(\alpha,z)\rangle \leq 0$ for all $z\in K$.
\end{enumerate}
Here, the notation $(v)_{m \dots n}$ stands for the vector $(v_m, \dots, v_n)$ with $v =(v_1,v_2,\dots,v_{d+k}) \in \mathbb{R}^{d+k}$.

Then for the step size sequences of the form $b(n)=\frac{1}{n(\log n)^p}, n\geq 2 $ with $0< p \leq 1$, the iterate sequence  
$\{\alpha_n\}$ is asymptotically tight.
\end{lem}

%This along with the assumption that $\liminf P(\theta_n \in G^s)=1$ will imply that 
%$\|w_n - \lambda(\theta_n)\| \to 0$ a.s. as well as 
%$(\theta_n,w_n) \to \bigcup_{\theta \in H^s} (\theta, \lambda(\theta))$.
%\end{comment}
%\begin{comment}
\section{Sample Complexity}
\label{sample}
It is easy to check that using the results in the previous section one can get a similar
probability estimate for sample complexity as in \cite[Chap. 4, Corollary 14]{borkar_book}. Note that here $T$ can be any positive real
number unlike in the lock-in probability calculation where we need to choose $T$ appropriately. Therefore
we can extend the sample complexity calculation for stochastic fixed point iteration
in the setting of  Markov iterate-dependent noise as follows:

%\begin{align}
%N_0:=\min \left{ n : \sum_{i=n_0 +1}^n} a(i) \geq \frac{(V(\theta_{n_0})-\epsilon)(T+1)}{\Delta/2}
%\end{align}
%more iterates to get within $\delta_{\epsilon} + \epsilon +  \Delta/2$ where $n_0$ should satisfy (\ref{n_0})
Consider the example as in \cite[p.~43]{borkar_book} that we describe below. Let $u(\theta) = \int f(\theta, y)\Gamma_\theta(dy)$ with
$u$ being a contraction, so that $\|u(\theta) - u(\theta')\| < \alpha \|\theta-\theta'\|$ for some $\alpha \in (0,1)$. Let
$\theta^*$ be the unique fixed point of $u(.)$. Let $T>0$ and $B$ be chosen to be  $\{\theta: \|\theta - \theta^*\| < r\}$
with $r \geq \frac{3\epsilon}{2}$. For the analysis, choose $r=\frac{3\epsilon}{2}$.
Then the sample complexity estimate can be stated as follows:
\begin{cor}
Let a desired accuracy $\epsilon >0$ and confidence $0<\gamma<1$ be given. Let $\bar{\theta}$ be the value at iteration $n_0$
with $n_0$ satisfying:
\begin{enumerate}
 \label{b_n}
 \item $n_0$ sufficiently large as in (\ref{large}), $s(n_0) < \frac{\hat{C}\epsilon^2}{4}$ and
$a(n_0) < \frac{\hat{C}\epsilon^2}{4}$ (Theorem 11 of \cite[Chap. 4]{borkar_book}).
 \item \begin{align}
        s(n_0) < \frac{c\epsilon^2}{\ln(\frac{4d}{\gamma})}.
       \end{align}
\end{enumerate}
Then on the event $\{\bar{\theta} \in B\}$, one needs
\begin{align}
N_0 := \min \left[n: \sum_{i=n_0 + 1}^{n}a(i) \geq \frac{(T+1)}{(1-e^{-(1-\alpha)T})}\right]-n_0\nonumber
\end{align}
more iterates to get within a distance $2\epsilon$ of $\theta^*$ with probability at least $1-\gamma$.
\end{cor}
\begin{rmk}
The results clearly show large vs. small step-size trade-off for non-asymptotic rate of convergence
well-known in the stochastic convex optimization literature \cite{nemro}.
For large step-size, the algorithm will make fast progress whereas the errors due
to noise/discretization  will be much higher simultaneously.
However,
our results show the quantitative estimate of this progress and the error. For the large
step-size case, $n_0$ satisfying the hypothesis in Corollary 6.1 will be 
higher whereas $N_0$ will be lower compared to the small-step size while the opposite is true
for the small step-size case. Therefore the optimal step-size should be somewhere in between.
\end{rmk}

However, it is not possible to calculate accurately the threshold $n_0$ as the constants such as 
$C, \hat{K}$ depend on $B$ which
indeed depends on $\theta^*$. If we consider some special cases where the range for $\theta^*$
is given although the actual $\theta^*$ is unknown, we can replace the terms involving constants in (\ref{large}) by a
single constant $M$. For those
cases the following analysis will be useful.

In the following we state an upper bound $N'_0$ of
$N_0 + n_0$ when $a(n) = \frac{1}{n^k}, \frac{1}{2} < k < 1$ under the following crucial assumption:
\begin{enumerate}[label=\textbf{(T\arabic*)}]
 \item $P(\theta_{n_0} \in B) =1$.
\end{enumerate}

%Let a desired accuracy $\epsilon >0$ and confidence $0<\gamma<1$ be given.
Let $\alpha = 0.9$. Under the assumptions made, the estimates of $n_0$ and $N'_0$ are
\footnotesize
\begin{align}
\nonumber
n_0  = \max ((\frac{M}{\epsilon})^{\frac{1}{k}},\ (\frac{M}{\epsilon(2k-1)})^{\frac{1}{2k-1}},(\frac{M}{\epsilon^2(2k-1)})^{\frac{1}{2k-1}},(\frac{M}{\epsilon})^{\frac{2}{k}}, \\
\ (\frac{M(\ln (\frac{1}{\gamma}))}{\epsilon^2(2k-1)})^{\frac{1}{2k-1}},\ (\frac{2Mk}{\epsilon(2k-1)})^{\frac{1}{2k-1}}), \\
N'_0 = \left(\left(n_0\right)^{\left(1-k\right)}+ 15.16(1-k)\right)^{\frac{1}{1-k}}.
\end{align}
\normalsize
Then from $N'_0$ onwards the iterates will be within $2\epsilon$ of $\theta^*$ with probability at least $1-\gamma$.
Note that the minimum value of the quantity $\frac{2(T+1)}{(1-e^{-(1-\alpha)T})}$ for $\alpha =0.9$ is  $15.16$.

To understand what should be the optimal step-size i.e. the value of $k$ for which $N'_0$ will be minimum, we plot $N'_0$ as
a function of $k$ for two different values of $M$ each with two different values of $\epsilon$ (see Figs. \ref{fig1} and \ref{fig2}).
%As discussed earlier,
%rather than looking into the actual values of the sample complexity, the shape of the curve is much more important.

\begin{figure*}
\centering
\subfigure[$\epsilon =0.01$]{\includegraphics[width = .49\linewidth]{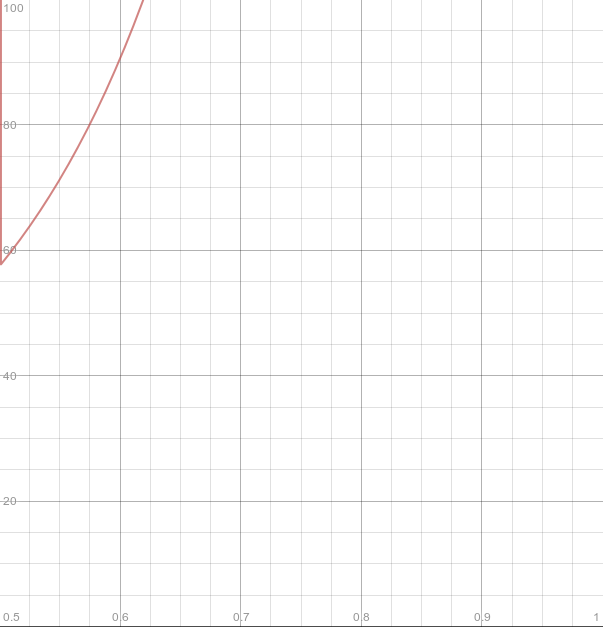}} \hfill
\subfigure[$\epsilon =0.001$]{\includegraphics[width = .49\linewidth]{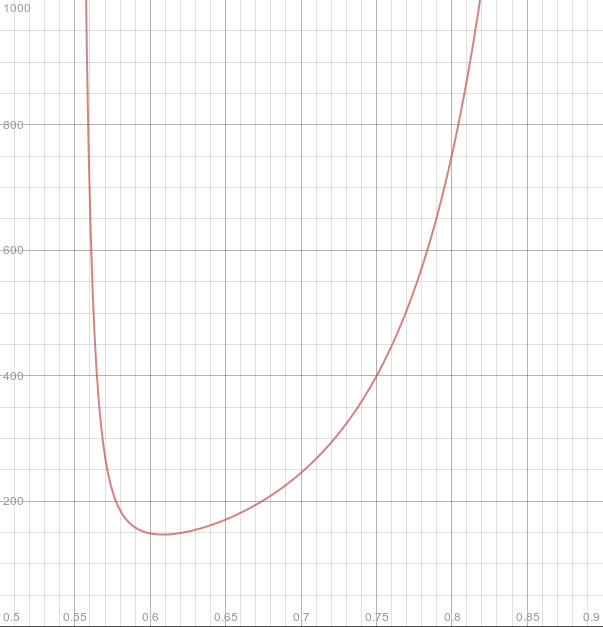}}
%\subfloat[$\epsilon =0.01$, $\alpha$ = 0.1]{\includegraphics[width = 4in]{/media/6428C06128C03438/Prasenjit_lockin/lock_in_Markov iterate-dependent/figure/6.662/01.png}}
\caption{Sample complexity vs. step-size parameter; $y: N'_0, x: k, \gamma =0.1$, $M$ = 1E-07}
\label{fig1}
\end{figure*}
\begin{figure*}
\centering
\subfigure[$\epsilon =0.01$]{\includegraphics[width = .49\linewidth]{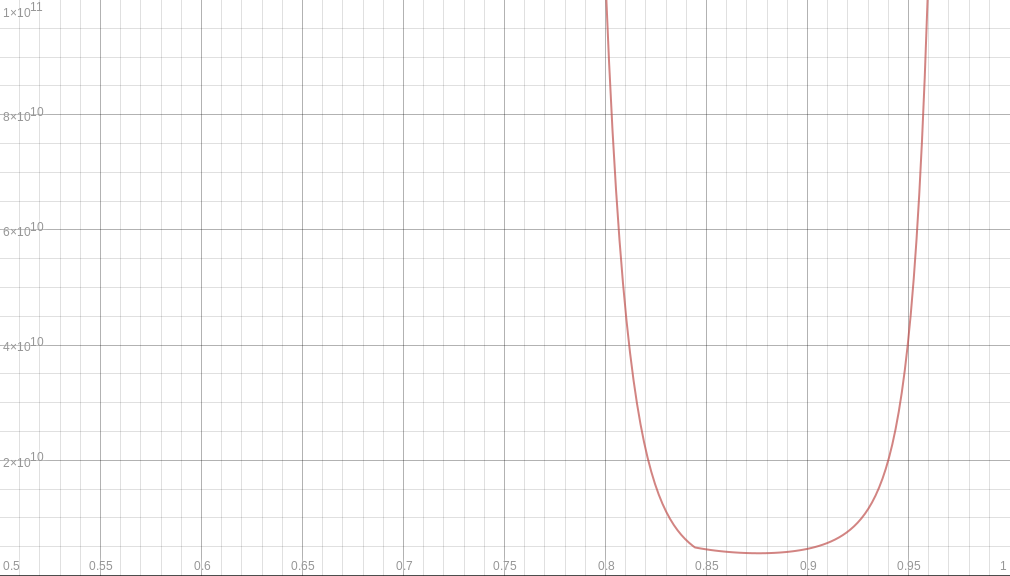}} \hfill
\subfigure[$\epsilon =0.001$]{\includegraphics[width = .49\linewidth]{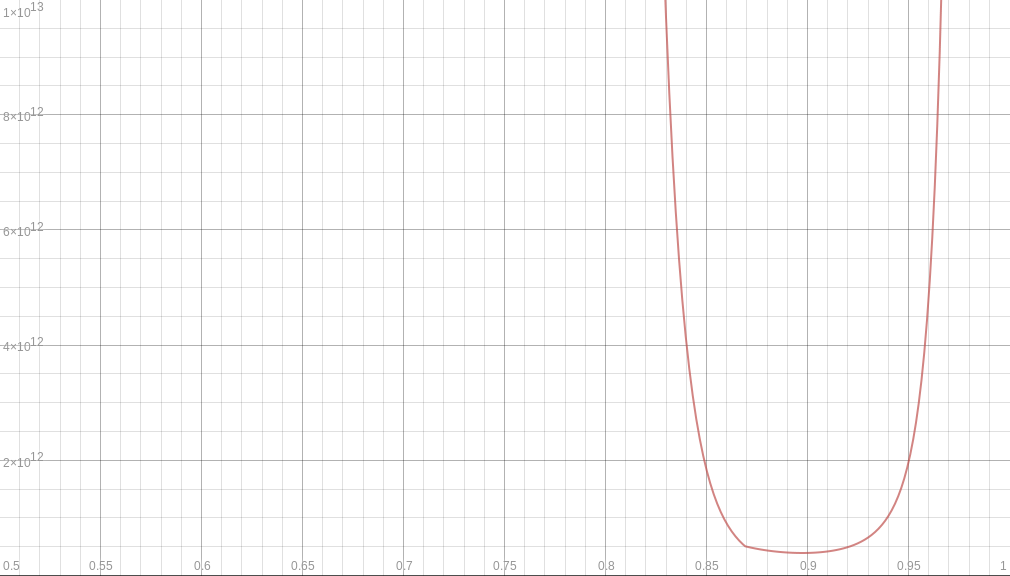}}
%\subfloat[$\epsilon =0.001$, $\alpha$ = 0.5]{\includegraphics[width = 4in]{/media/6428C06128C03438/Prasenjit_lockin/lock_in_Markov iterate-dependent/figure/9.43/001.png}}
%\\subfloat[fig 3]{\includegraphics[width = 3in]{something}}
%\subfloat[fig 4]{\includegraphics[width = 3in]{something}}
\caption{Sample complexity vs. step-size parameter; $y: N'_0, x: k, \gamma =0.1$, $M$ = 100}
\label{fig2}
\end{figure*}
From the graph it is clear that for large values of $M$, the optimal $k$ is biased towards $1$ whereas for small values of $k$ it
is biased toward $\frac{1}{2}$. The reason is that for large $M$, with large step-size, $n_0$ will be much higher although $N'_0 -n_0$
is small whereas
with very small $M$, even if we use large step-size, $n_0$ will not be large, thus one can take advantage of
$N'_0 -n_0$ being small.

\section{Conclusion}
In this paper, we describe asymptotic and non-asymptotic convergence analysis of
stochastic approximation recursions with  Markov iterate-dependent noise using the
lock-in probability framework.
Our results show that we are able to recover the same bound available for lock-in probability in the literature
for the much stronger i.i.d noise case. Such results are used to calculate sample complexity estimate
%non-asymptotic rate of convergence ``involving''
of such stochastic approximation recursions which are then used for predicting the optimal step size.
Moreover, our results are extremely useful to prove almost sure convergence to specific attractors
in cases where asymptotic tightness
of the iterates can be proved easily.
%Moreover, using our results, we provide a statement similar to
%Kushner-Clarke Lemma however without the assumption of the stability of the iterates.
An interesting future direction will be to extend this
analysis to two-timescale scenarios with coupling both ways between the recursions and, both with and without Markov iterate-dependent noise.

\appendices
%and then carry on using the \section and \subsection commands, as above.

\section{Proof of conditional and maximal version of Azuma's inequality}
Let $P_B$ denote probability measure defined by $ P_B(A)=\frac{P(A\cap B)}{P(B)}$ where $B \in \mathcal{F}_1$.
If we can show that with this new probability measure $\{S_n\}$ is a martingale, then we can follow the steps in
\cite[(3.30), p~227]{mcdarmiad} to conclude the proof.

Let us denote by $ E_B$ the expectation with respect to $ P_B$. Clearly,
$E_B(X)= \frac{\int_BX dP}{P(B)}$.
Let $G \in \mathcal{F}_n$.
Now,
\begin{align}
\nonumber
\int_G E_B[S_{n+1}|&\mathcal{F}_n]dP_B = E_B[E_B[I_G S_{n+1}|\mathcal{F}_n]] \nonumber \\
                   &=E_B[I_G S_{n+1}] = \frac{E[I_{G \cap B} S_{n+1}]}{P(B)} \nonumber \\
                   &=\frac{E[I_{G \cap B}E[S_{n+1}|\mathcal{F}_n]]}{P(B)} \nonumber \\
                   &=\frac{E[I_{G \cap B} S_n]}{P(B)} = \int_G S_n dP_B. \nonumber
\end{align}
\section{General discrete Gronwall inequality}
Let $\{\theta_n,n\geq 0\}$ (respectively $\{a_n, n\geq 0\})$ be non-negative (respectively positive)
sequences, $L \geq 0$ and $f(n)$ be a increasing function of $n$ such that for all $n$
\begin{align}
\theta_{n+1} \leq f(n) + L (\sum_{m=0}^{n}a_m \theta_m)\nonumber.
\end{align}
Then for $T_n = \sum_{m=0}^{n}a_m$,
\begin{align}
\theta_{n+1} \leq f(n) e^{LT_n} \nonumber
\end{align}
\begin{proof}
Similar to the proof of Lemma 8 in Appendix B of \cite{borkar_book}
\end{proof}

\begin{IEEEbiography}[{\includegraphics[width=1in,height=1.25in,clip,keepaspectratio]{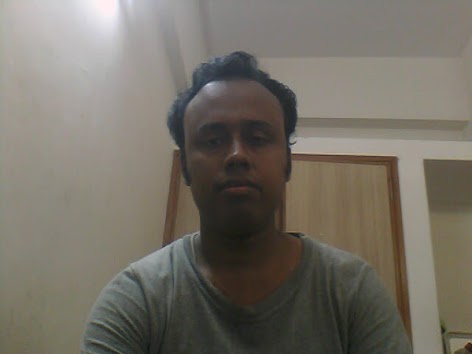}}]%
{Prasenjit Karmakar}

received the Master’s and Ph.D.
degrees in computer science and automation
from the Indian Institute of Science in 20012 and
2018, respectively.
He will soon join as a postdoctoral researcher at
EE Technion. His research interests are in reinforcement learning, stochastic
approximation and applied probability.
\end{IEEEbiography}

\begin{IEEEbiography}[{\includegraphics[width=1in,height=1.25in,clip,keepaspectratio]{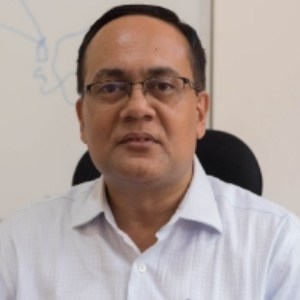}}]%
{Shalabh Bhatnagar}

received the Bachelor’s (Hons.) degree in physics from the University of Delhi, Delhi, India, in 1988 and the Master’s and Ph.D. degrees in electrical engineering
from the Indian Institute of Science, Bangalore,
in 1992 and 1997, respectively.
He is a Professor and Chairman with the Department of
Computer Science and Automation as well as Robert Bosch Centre for Cyber Physical Systems, Indian Institute of Science, Bangalore. His research interests are in stochastic approximation theory and
stochastic optimization and control, as well as
applications in communication, wireless, and vehicular traffic networks.
\end{IEEEbiography}

\end{document}